[1994/12/01]% LaTeX date must December 1994 or later
\documentclass[a4paper]{article}
\title{On the Lattice Packings and Coverings of the Plane with Convex Quadrilaterals}
\author{Ktrati Sriamorn}

\usepackage{amsmath,amsthm,amssymb,graphicx}

\chardef\bslash=`\\ % p. 424, TeXbook
\newtheorem{thm}{Theorem}[section]
\newtheorem{cor}[thm]{Corollary}
\newtheorem{lem}[thm]{Lemma}

\theoremstyle{definition}

\theoremstyle{remark}
\newtheorem{rem}{Remark}[section]

\newcommand{\eval}[2][\right]{\relax
  \ifx#1\right\relax \left.\fi#2#1\rvert}

\begin{document}
\maketitle
\markboth{ On he Lattice Packings and Coverings with Convex Quadrilaterals}
{On the Lattice Packings and Coverings with Convex Quadrilaterals}
\renewcommand{\sectionmark}[1]{}

\begin{abstract}
It is well known that the lattice packing density and the lattice covering density of a triangle are $\frac{2}{3}$ and $\frac{3}{2}$ respectively \cite{fary}. We also know that the lattices that attain these densities both are unique. Let $\delta_{L}(K)$ and $\vartheta_{L}(K)$ denote the lattice packing density and the lattice covering
density of $K$, respectively. In this paper, I study the lattice packings and coverings for a special class of convex disks, which includes all triangles and convex quadrilaterals. In particular, I determine the densities $\delta_{L}(Q)$ and $\vartheta_{L}(Q)$, where $Q$ is an arbitrary convex quadrilateral. Furthermore, I also obtain all of lattices that attain these densities. Finally, I show that $\delta_{L}(Q)\vartheta_{L}(Q)\geq 1$ and $\frac{1}{\delta_{L}(Q)}+\frac{1}{\vartheta_{L}(Q)}\geq 2$, for each convex quadrilateral $Q$.
\end{abstract}

\section{Introduction and preliminaries}\label{intro}
An $n$-dimensional \emph{convex body} is a compact convex subset of $\mathbb{R}^n$ with an interior point. A 2-dimensional convex body is called a \emph{convex disk}. The $n$-dimension measure of a set $S$ will be denoted by $|S|$. The closure and the interior of $S$ will be denoted by $\overline{S}$ and $Int(S)$, respectively. The cardinality of $S$ is denoted by $card\{S\}$

For any $n$ independent vectors $\mathbf{v}_1,\ldots,\mathbf{v}_n$ in $\mathbb{R}^n$, the \emph{lattice} generated by $\mathbf{v}_1,\ldots,\mathbf{v}_n$ is the set of vectors
$$\{ k_1\mathbf{v}_1 + \cdots + k_n\mathbf{v}_n : k_1,\ldots,k_n \text{~are integers} \},$$
denoted by $\mathcal{L}(\mathbf{v}_1,\ldots,\mathbf{v}_n)$.

A family $\mathcal{F}=\{K_1,K_2,\ldots\}$ of convex bodies is called a \emph{covering} of a domain $D\subseteq \mathbb{R}^n$ provided
$D\subseteq\bigcup_{i}K_i,$
and call $\mathcal{F}$ a \emph{packing} of $D$ if the interiors are mutually disjoint and $\bigcup_{i}K_i\subseteq D$.  A family $\mathcal{F}$ which is both a packing and a covering of $D$ is called a \emph{tiling} of $D$. Any compact set which admits a tiling of $\mathbb{R}^n$ is called a \emph{tile}.

The \emph{upper} and \emph{lower density} of a family $\mathcal{F}=\{K_1,K_2,\ldots\}$ of convex bodies with respect to a bounded domain $D$ are defined as
$$d_{+}(\mathcal{F},D)=\frac{1}{|D|}\sum_{K\in\mathcal{F}, K\cap D\neq\emptyset}|K|,$$
and
$$d_{-}(\mathcal{F},D)=\frac{1}{|D|}\sum_{K\in\mathcal{F}, K\subset D}|K|.$$
We define the \emph{upper} and \emph{lower density} of the family $\mathcal{F}$ by
$$d_{+}(\mathcal{F})=\limsup_{r\rightarrow \infty} d_{+}(\mathcal{F},rB^n),$$
and
$$d_{-}(\mathcal{F})=\liminf_{r\rightarrow \infty} d_{-}(\mathcal{F},rB^n),$$
where $B^n$ denotes the unit ball in $\mathbb{R}^2$, centered at the origin.

The \emph{packing density} $\delta(K)$ of convex body $K$ is defined by the formula
$$\delta(K)=\sup_{\mathcal{F}}d_{+}(\mathcal{F}),$$
the supremum being taken over all packings $\mathcal{F}$ of $\mathbb{R}^n$ with congruent copies of $K$.
The \emph{covering density} $\vartheta(K)$ of convex body $K$ is defined by the formula
$$\vartheta(K)=\inf_{\mathcal{F}}d_{-}(\mathcal{F}),$$
the infimum being taken over all coverings $\mathcal{F}$ of $\mathbb{R}^n$ with congruent copies of $K$.

One may consider arrangements of translated copies of $K$ only, or just lattice arrangements of translates of $K$. In these cases, the corresponding densities assigned to $K$ by analogous definitions are: the \emph{translative} packing and covering density of $K$, denoted by $\delta_{T}(K)$ and $\vartheta_{T}(K)$, and the \emph{lattice} packing and covering density of $K$, denoted by $\delta_{L}(K)$ and $\vartheta_{L}(K)$, respectively.

Let $K$ be an $n$-dimensional convex body. Suppose that $x\in\mathbb{R}^n$ and $X$ is a discrete subset of $\mathbb{R}^n$, we define $K+x=\{y+x : y\in K\}$,
and denote by $K+X$ the family $\{K+x\}_{x\in X}$.

An \emph{optimal packing lattice} of $K$ is a lattice $\Lambda$ which $K+\Lambda$ is a packing of $\mathbb{R}^n$ with density $\delta_{L}(K)$. Denote by $\Delta(K)$ the collection of all optimal packing lattice of $K$. Similarly, An \emph{optimal covering lattice} of $K$ is a lattice $\Lambda$ which $K+\Lambda$ is a covering of $\mathbb{R}^n$ with density $\vartheta_{L}(K)$. Denote by $\Theta(K)$ the collection of all optimal covering lattice of $K$.

Let $\mathcal{K}^n$ denote the collection of all $n$-dimensional convex bodies. Let $K_1+K_2$ denote the \emph{Minkowski sum} of $K_1$ and $K_2$ defined by
$$K_1+K_2 = \{x_1+x_2 :~x_i\in K_i\},$$
let $\|\cdot\|^*$ denote the \emph{Hausdorff metric} on $\mathcal{K}^n$ defined by
$$\|K_1,K_2\|^*=\min \{r:~K_1\subset K_2+rB^n,~K_2\subset K_1+rB^n\},$$
and let $\{\mathcal{K}^n,\|\cdot\|^*\}$ denote the space of $\mathcal{K}^n$ with metric $\|\cdot\|^*$. It is easy to see that, for $\lambda_i\in\mathbb{R}$ and $K_i\in\mathcal{K}^n$,
$$\lambda_1K_1+\lambda_2K_2+\cdots +\lambda_mK_m\in\mathcal{K}^n.$$
In certain sense, the space $\mathcal{K}^n$ has linear structure. Blaschke selection theorem guarantees the local compactness of $\{\mathcal{K}^n,\|\cdot\|^*\}$. It is easy to show that all $\delta(K),\delta_{T}(K),\delta_{L}(K),\vartheta(K),\vartheta_{T}(K)$ and $\vartheta_{L}(K)$ are bounded continuous functions defined on $\{\mathcal{K}^n,\|\cdot\|^*\}$.

Define $\omega : \mathcal{K}^2 \rightarrow \mathbb{R}^2$ by $\omega(K)=(\delta(K),\vartheta(K))$ for every $K\in \mathcal{K}^2$. By continuity of each of the real-valued functions $\delta$ and $\vartheta$, the function $\omega$ is continuous. Let $\Omega= \omega(\mathcal{K}^2)$. Similarly, we can define the sets $\Omega_{T}$ and $\Omega_{L}$ by replacing the function $\omega=(\delta,\vartheta)$ by $\omega_{T}=(\delta_{T},\vartheta_{T})$ and by $\omega_{L}=(\delta_{L},\vartheta_{L})$, respectively.

Since the Minkowski sum of two \emph{centrally symmetric} sets is a centrally symmetric set, the analogous statements hold for the space $\mathcal{C}^n$ of centrally symmetric $n$-dimensional convex bodies, and to the corresponding sets $\Omega^*=\omega(\mathcal{C}^2)$, $\Omega^*_{T}=\omega_{T}(\mathcal{C}^2)$ and $\Omega^*_{L}=\omega_{L}(\mathcal{C}^2)$. It is well known that $\vartheta_T(C)=\vartheta_L(C)$, for all centrally symmetric convex disks $K$ \cite{fejes}. This immediately follows that $\Omega^*_{T}=\Omega^*_{L}$.

The question of describing explicitly the sets $\Omega$, $\Omega_{T}$, $\Omega_{L}$, $\Omega^*$, and $\Omega_{L}^*$ remains open. In 2001 Ismailescu \cite{ismailescu} proved that for each centrally symmetric convex disk $C$,
$$1-\delta_{L}(C)\leq\vartheta_{L}(C)-1\leq\frac{5}{4}\sqrt{1-\delta_{L}(C)}.$$
These inequalities can be expressed as : the set $\Omega_{L}^*=\Omega_{T}^*$ lies between the line $x+y=2$ and the curve $y=1+\frac{5}{4}\sqrt{1-x}$. A recent paper of Ismailescu and Kim \cite{ismailescu kim} showed that $\delta_{L}(C)\vartheta_{L}(C)\geq 1$ for every centrally symmetric convex disk $K$, which is stronger than Ismailescu's inequality $\delta_{L}(C)+\vartheta_{L}(C)\geq 2$ mentioned above. It is still unknown whether these inequalities hold for any (non-symmetric) convex disks. However, I will show later that these inequalities hold for convex quadrilaterals.

We note that if one can prove that the inequality $\frac{1}{\delta_{L}(C)}+\frac{1}{\vartheta_{L}(C)}\leq 2$ holds for every \emph{centrally symmetric} convex disk $C$, this would represent an improvement over the inequality $\delta_{L}(C)\vartheta_{L}(C)\geq 1$. Unfortunately, this still is an open problem. In general case (not necessarily symmetrical), it is obvious that there exist convex disks $K$ such that the inequality $\frac{1}{\delta_{L}(K)}+\frac{1}{\vartheta_{L}(K)}\leq 2$ dose not hold. In fact, $\frac{1}{\delta_{L}(T)}+\frac{1}{\vartheta_{L}(T)}=\frac{3}{2}+\frac{2}{3}> 2$, for triangles $T$. Moreover, we will see later in this paper that $\frac{1}{\delta_{L}(Q)}+\frac{1}{\vartheta_{L}(Q)}\geq 2$, for every convex quadrilaterals $Q$.

Let $f(x)$ be a convex and non-increasing continuous function with $f(0)=1$ and $f(1)\geq 0$. Throughout this paper, we define the convex disk
$$K_f=\{(x,y) : 0\leq x\leq 1~,~0\leq y\leq f(x)\}.$$
Denote by $C_{f}$ the curve $\{(x,f(x)) :~0< x<1\}\cup \{(1,y) :~0<y\leq f(1)\}$.

\begin{figure}[!ht]
  \centering
    \includegraphics[scale=.80]{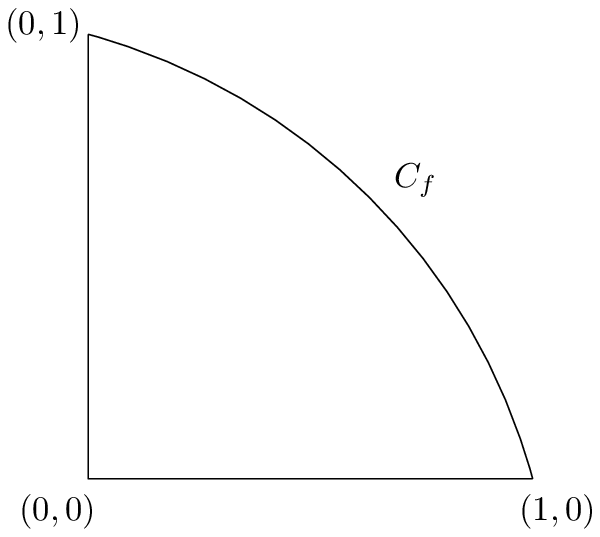}
   \caption{$K_f$}
\end{figure}
A very recent paper of Xue and Kirati \cite{xuefei} proved the following result

\begin{thm}\label{transequallattice}
$\vartheta_T(K_f)=\vartheta_L(K_f)$.
\end{thm}

\section{Main Results}
Let $K_{x,y}$ denote the quadrilateral with vertices $(0,1),(0,0),(1,0)$ and $(x,y)$. Denote by $D$ the set of all points $(x,y)$ that satisfy $0\leq x\leq y \leq 1$ and $x+y\geq 1$. Let $Q$ be an arbitrary convex quadrilateral. One can easily show that there exists $(x,y)\in D$ such that $Q$ and $K_{x,y}$ are affinely equivalent.

\begin{figure}[ht]
\begin{minipage}[b]{0.45\linewidth}
\centering
\includegraphics[width=\textwidth]{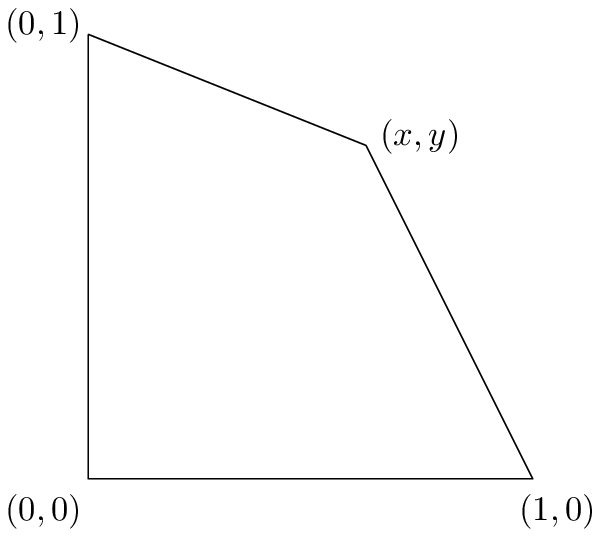}
\caption{$K_{x,y}$}
\end{minipage}
\hspace{0.5cm}
\begin{minipage}[b]{0.45\linewidth}
\centering
\includegraphics[width=\textwidth]{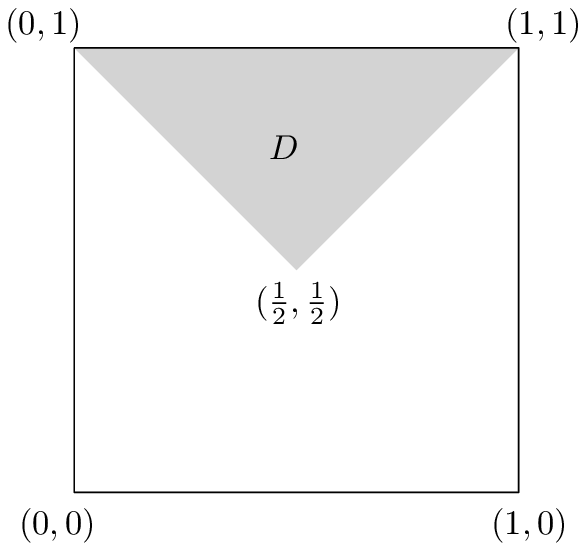}
\caption{$D$}
\end{minipage}
\end{figure}

It is well known that $\delta_{L}$ and $\vartheta_{L}$ are affinely invariant. In order to determine $\delta_{L}(Q)$, $\vartheta_{L}(Q)$, $\Delta(Q)$ and $\Theta(Q)$, we may assume, without loss of generality, that $Q=K_{x,y}$, where $(x,y)\in D$. We define
$$\delta_{L}(x,y)=\delta_{L}(K_{x,y}), ~\vartheta_{L}(x,y)=\vartheta_{L}(K_{x,y}),$$
and
$$\Delta(x,y)=\Delta(K_{x,y}), ~\Theta(x,y)=\Theta(K_{x,y}).$$

The main results are as follows

\begin{thm}\label{main theorem}
Suppose that $(x,y)\in D$, then
$$\delta_{L}(x,y)=\frac{2y(x+y)}{4y+x-1}$$
and
$$
\vartheta_{L}(x,y)=
\begin{cases}
\frac{3(x+y)(1-x)}{2y} & x\leq\frac{1}{3},\\
\frac{2(x+y)}{y(1+3x)} & x\geq\frac{1}{3} \text{~~and~~} y\geq\frac{2}{3},\\
\frac{(x+y)(4(1-x)(1-y)-xy)}{2(x(1-x)+y(1-y)-xy)} & y<\frac{2}{3}.\\
\end{cases}
$$
\end{thm}

\begin{thm}\label{main theorem_1}
Suppose that $(x,y)\in D$, then
$$
\Delta(x,y)=
\begin{cases}
\{\mathcal{L}((0,1),(1,t)),~ \mathcal{L}((t,1),(1,0)) :~ t\in(0,1]\} &x=y=1,\\
\left\{\mathcal{L}\left(\left(\frac{y-1}{2y},1\right),\left(\frac{3y-1}{2y},\frac{1}{2}\right)\right),\right.\\ \left.
~~~~~~~~~~~~~~~~~~~~\mathcal{L}\left(\left(-\frac{1}{2},1\right),\left(\frac{1}{2},\frac{3y-1}{2y}\right)\right) \right\} & \frac{1}{2}<x= y<1,\\
\left\{\mathcal{L}\left(\left(\frac{x-1}{2y},1\right),\left(1-\frac{1-x}{2y},\frac{1}{2}\right)\right)\right\} & otherwise,\\
\end{cases}
$$
and
$$
\Theta(x,y)=
\begin{cases}
\left\{\mathcal{L}\left( \left(-\frac{1}{3},\frac{2y}{3(1-x)}\right),\left(\frac{1}{3},\frac{y}{3(1-x)}\right)\right)  \right\} & x\leq \frac{1}{3},\\
\left\{\mathcal{L}\left(\left(\frac{x-1}{2},y \right),\left(x,\frac{y}{2}\right) \right) \right\} & \frac{1}{3}< x<y , \frac{2}{3}\leq y,\\
\left\{\mathcal{L}\left( \left(\frac{y(1+y-2x)-2(1-x)^2}{4(1-x)(1-y)-xy},\frac{(2-y)(1-x)-y^2}{4(1-x)(1-y)-xy}\right), \right. \right.\\ \left.\left. ~~~~~~~~~\left(\frac{x(2(1-x)-y)}{4(1-x)(1-y)-xy},\frac{y(2(1-y)-x)}{4(1-x)(1-y)-xy}\right)\right) \right\} & y< \frac{2}{3},\\
\left\{\mathcal{L}\left(\left(\frac{y-1}{2},y \right),\left(y,\frac{y}{2}\right)\right),\mathcal{L}\left(\left(-\frac{ y}{2},\frac{1+y}{2} \right),\left(\frac{y}{2},y\right)\right)\right\} &\frac{2}{3}<x=y<1,\\
\left\{\mathcal{L}\left(\left(-\frac{t}{2},\frac{1+t}{2}\right),\left(1-t,t\right)\right) :~t\in\left[\frac{1}{3},\frac{2}{3}\right]\right\} & x=y=\frac{2}{3},\\
\{\mathcal{L}((0,1),(1,t)),~ \mathcal{L}((t,1),(1,0)) :~ t\in(0,1]\} &x=y=1.
\end{cases}
$$
\end{thm}

\begin{cor}\label{cor_optimal}
Suppose that $(x,y)\in D$, then
$$
card\{\Delta(x,y)\}=
\begin{cases}
\mathfrak{c} & x=y=1\\
2 & \frac{1}{2}<x= y<1,\\
1 & otherwise,\\
\end{cases}
$$
and
$$
card\{\Theta(x,y)\}=
\begin{cases}
\mathfrak{c} & x=y=\frac{2}{3} \text{~~or~~} x=y=1,\\
2 & \frac{2}{3}<x=y<1,\\
1 & otherwise,
\end{cases}
$$
where $\mathfrak{c}$ is the cardinality of the real numbers.
\end{cor}

\begin{cor}\label{main cor}
For every convex quadrilaterals $Q$, we have
$$\delta_{L}(Q)\vartheta_{L}(Q)\geq 1$$
and
$$\frac{1}{\delta_{L}(Q)}+\frac{1}{\vartheta_{L}(Q)}\geq 2.$$
\end{cor}

\begin{cor}\label{main cor_2}
For every convex quadrilaterals $Q$, we have
$$\delta_{L}(Q)+\vartheta_{L}(Q)\geq 2$$
and
$$\vartheta_{L}(Q)\leq 1+\frac{5}{4}\sqrt{1-\delta_{L}(Q)}.$$
\end{cor}

\begin{rem}
It is well known that $\delta_T(K)=\delta_L(K)$, for every convex disks $K$ (see \cite{rogers_2} or \cite{fejes_2}). Furthermore, according to Theorem \ref{transequallattice}, we know that $\vartheta_T(Q)=\vartheta_L(Q)$, for each convex quadrilateral $Q$. Therefore, the above statements also hold for $\delta_T$ and $\vartheta_T$.
\end{rem}

\section{The lattice packings and coverings of $K_f$}
In this section, for convenience, we assume that $K_f$ is not a unit square.

For any real numbers $x,x',y,y'$, let $L(x,y,x',y')$ denote the line segment between $(x,y)$ and $(x',y')$. Suppose that $0\leq x\leq1$, denote by $S^{f}(x)$ the polygon bounded by $L(0,0,0,1)$, $L(0,1,x,1)$, $L(x,1,x,f(x))$, $L(x,f(x),1,f(x))$, $L(1,f(x),1,0)$ and $L(1,0,0,0)$. For $0\leq x_1\leq x_2\leq1$, denote by $S_{f}(x_1,x_2)$ the polygon bounded by $L(0,0,0,f(x_1))$, $L(0,f(x_1),x_1,f(x_1))$, $L(x_1,f(x_1),x_1,f(x_2))$, $L(x_1,f(x_2),x_2,f(x_2))$, $L(x_2,f(x_2),x_2,0)$ and $L(x_2,0,0,0)$.

\begin{figure}[ht]
\begin{minipage}[b]{0.45\linewidth}
\centering
\includegraphics[width=\textwidth]{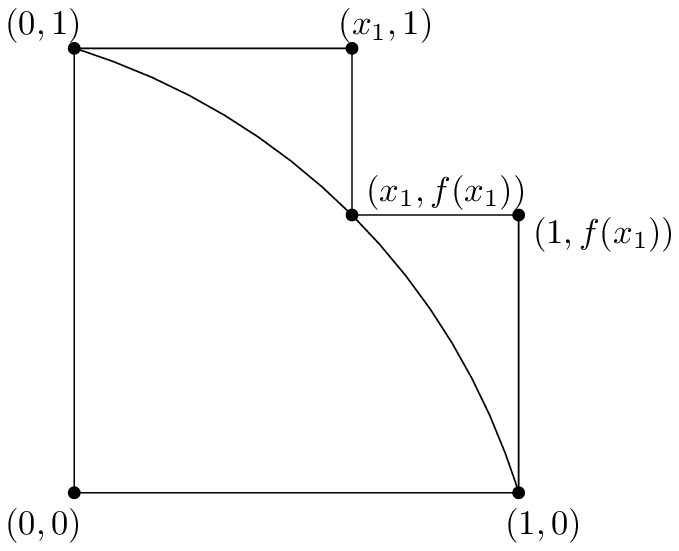}
\caption{$S^{f}(x_1)$}
\end{minipage}
\hspace{0.5cm}
\begin{minipage}[b]{0.45\linewidth}
\centering
\includegraphics[width=\textwidth]{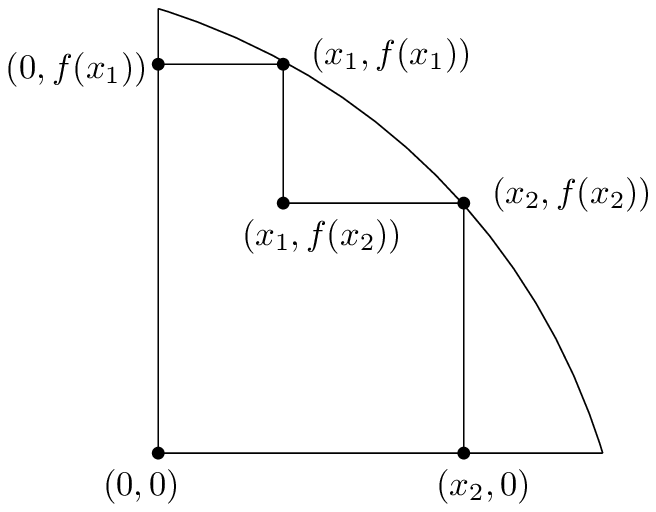}
\caption{$S_{f}(x_1,x_2)$}
\end{minipage}
\end{figure}

Define
$$A^{f}= \min \{|S^{f}(x)| : 0\leq x\leq 1\},$$
and
$$A_{f}= \max \{|S_{f}(x_1,x_2)| : 0\leq x_1\leq x_2\leq 1\}.$$
Let
$$\mathcal{X}^{f}=\{x : 0\leq x\leq 1, |S^{f}(x)|=A^{f} \},$$
and
$$\mathcal{X}_{f}=\{(x_1,x_2) : 0\leq x_1\leq x_2\leq1, |S_{f}(x_1,x_2)|=A_{f} \}.$$
The paper of Xue and Kirati \cite{xuefei} showed that
\begin{thm}\label{transequallattice_witharea}
$\vartheta_T(K_f)=\vartheta_L(K_f)=\frac{|K_f|}{A_{f}}$.
\end{thm}
Furthermore, from the results of the paper, we can also deduce that
\begin{lem}\label{stairdown}
For any (2-dimension) lattice $\Lambda$, if $\Lambda\in \Theta(K_f)$ (i.e., $K_f+\Lambda$ is a covering of $\mathbb{R}^2$ with density $\vartheta_L(K_f)$), then there exist $0\leq x_1\leq x_2\leq1$ such that $S_{f}(x_1,x_2)\subset K_f$, $S_{f}(x_1,x_2)+\Lambda$ is a tiling of $\mathbb{R}^2$, and $|S_{f}(x_1,x_2)|=A_{f}$
\end{lem}

For any $0\leq x_1\leq x_2\leq 1$, it is easy to see that $S_{f}(x_1,x_2)+\mathcal{L}((x_1-x_2,f(x_1)),(x_1,f(x_2)))$ is a tiling of $\mathbb{R}^2$ (see Figure \ref{lambdadown}). Define $\Lambda_{f}$ by
\begin{equation}\label{lambdadownlattice}
\Lambda_{f}(x_1,x_2)=\mathcal{L}((x_1-x_2,f(x_1)),(x_1,f(x_2))).
\end{equation}
Since $S_{f}(x_1,x_2)\subset K_f$, we obtain $K_f+\Lambda_{f}(x_1,x_2)$ is a covering of $\mathbb{R}^2$ with density $\frac{|K_f|}{|S_{f}(x_1,x_2)|}$. By Theorem \ref{transequallattice_witharea}, we know that $\Lambda_{f}$ can be seen as a mapping from $\mathcal{X}_{f}$ to $\Theta(K_f)$. From Lemma \ref{stairdown}, we obtain $\Lambda_{f}$ is surjective. To show that $\Lambda_{f}$ is injective, we first prove the following lemma.

\begin{figure}[!ht]
  \centering
    \includegraphics[scale=.8]{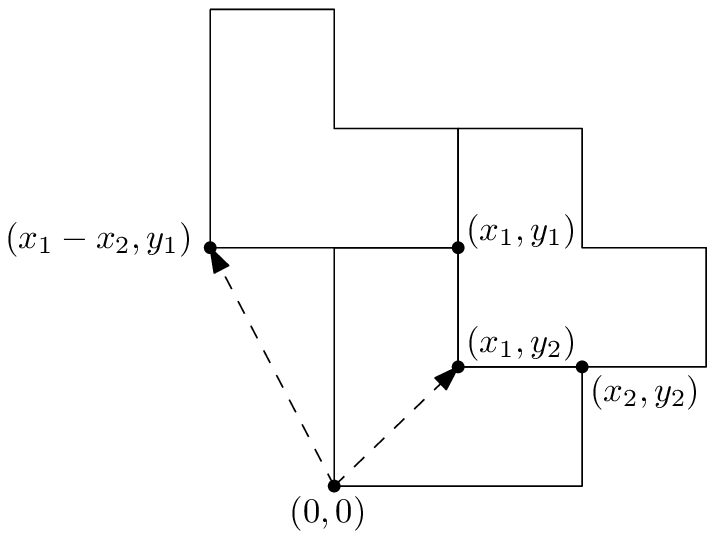}
   \caption{$\mathcal{L}((x_1-x_2,y_1),(x_1,y_2))$}\label{lambdadown}
\end{figure}

\begin{lem}\label{property of p1p2}
Suppose that $(x_1,x_2)\in \mathcal{X}_{f}$. We have
$$2x_1\geq x_2~,~2x_2\geq 1+x_1,$$
and
$$2f(x_2)\geq f(x_1)~,~2f(x_1)\geq 1+f(x_2).$$
\end{lem}
 \begin{proof}
 Since $K_f$ is not a square, one can easily see that $0<x_1<x_2$. By the definition of $\mathcal{X}_{f}$, we know that for all sufficiently small $\epsilon>0$,
 \begin{align*}
 x_1f(x_1)+(x_2-x_1)f(x_2)&= |S_{f}(x_1,x_2)|\\
 &\geq |S_{f}(x_1-\epsilon,x_2)|\\
 &=(x_1-\epsilon)f(x_1-\epsilon)+(x_2-x_1+\epsilon)f(x_2).
 \end{align*}
 Since $f$ is convex and $f(0)=1$, we have
 $$f(x_1-\epsilon)\geq f(x_1)+\frac{1-f(x_1)}{x_1}\epsilon.$$
 Therefore
 $$x_1f(x_1)+(x_2-x_1)f(x_2) \geq (x_1-\epsilon)\left(f(x_1)+\frac{1-f(x_1)}{x_1}\epsilon\right)+(x_2-x_1+\epsilon)f(x_2)$$
 From this, we can easily obtain
 $$2f(x_1)\geq 1+f(x_2)-\frac{1-f(x_1)}{x_1}\epsilon.$$
On allowing $\epsilon$ to end to zero, we get
$$2f(x_1)\geq 1+f(x_2).$$
By symmetry, we have
$$2x_2 \geq 1+x_1.$$
On the other hand,
 \begin{align*}
 x_1f(x_1)+(x_2-x_1)f(x_2)&= |S_{f}(x_1x_2)|\\
 &\geq |S_{f}(x_1,x_2-\epsilon)|\\
 &=x_1f(x_1)+(x_2-\epsilon-x_1)f(x_2-\epsilon).
 \end{align*}
 Because $f$ is convex, we have
 $$f(x_2-\epsilon)\geq f(x_2)+\frac{f(x_1)-f(x_2)}{x_2-x_1}\epsilon.$$
 Therefore
 $$ x_1f(x_1)+(x_2-x_1)f(x_2)\geq x_1f(x_1)+(x_2-\epsilon-x_1)\left(f(x_2)+\frac{f(x_1)-f(x_2)}{x_2-x_1}\epsilon\right).$$
From this, we can deduce that
$$2f(x_2)\geq f(x_1)-\frac{f(x_1)-f(x_2)}{x_2-x_1}\epsilon,$$
and hence
$$2f(x_2)\geq f(x_1).$$
By symmetry, we immediately get
$$2x_1\geq x_2.$$
This completes the proof.
 \end{proof}

\begin{thm}\label{thetabijection}
$\Lambda_{f}$ is a bijection from $\mathcal{X}_{f}$ to $\Theta(K_f)$
\end{thm}
\begin{proof}
Let $(x_1,x_2),(x'_1,x'_2)\in\mathcal{X}_{f}$.
Since $K_f$ is not a square, one can show that
$$(x_1,f(x_2))\in Int(K_f) ~\text{and}~ (x'_1,f(x'_2))\in Int(K_f).$$
By Lemma \ref{property of p1p2}, we know that
$$x_1+x_2\geq 1~,~f(x_1)+f(x_2)\geq 1$$
and
$$2x_1\geq x_2~,~2f(x_2)\geq f(x_1).$$
This can be deduced that (see Figure \ref{latticepoint})
$$\Lambda_{f}(x_1,x_2)\cap Int(K_f)=\{(x_1,f(x_2))\}.$$
By the same reason, we have
$$\Lambda_{f}(x'_1,x'_2)\cap Int(K_f)=\{(x'_1,f(x'_2))\}.$$
Now suppose that $(x'_1,x'_2) \neq (x_1,x_2)$. if $(x_1,f(x_2))\neq (x'_1,f(x'_2))$, then we have $\Lambda_{f}(x_1,x_2)\neq\Lambda_{f}(x'_1,x'_2)$. if $(x_1,f(x_2))= (x'_1,f(x'_2))$, then we may assume, without loss of generality, that $x_2<x'_2$, i.e., $S_{f}(x_1,x_2)\subsetneq S_{f}(x'_1,x'_2)$, and hence $|S_{f}(x_1,x_2)|< |S_{f}(x'_1,x'_2)|$. This is impossible. It follows that $\Lambda_{f}$ is an injection.
\end{proof}

\begin{figure}[!ht]
  \centering
    \includegraphics[scale=.8]{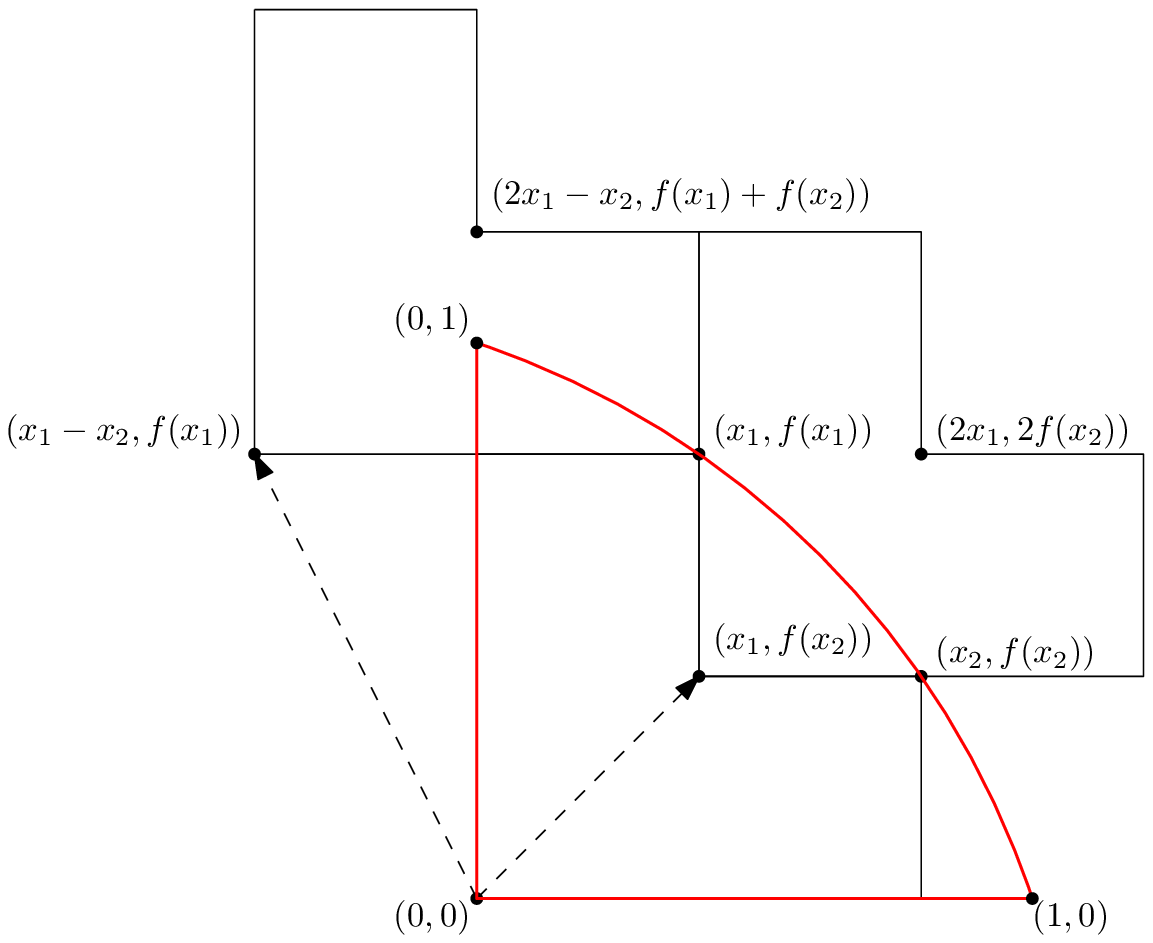}
   \caption{$\Lambda_{f}(x_1,x_2)\cap Int(K_f)=\{(x_1,f(x_2))\}$}\label{latticepoint}
\end{figure}

Now we determine $\delta_L(K_f)$. By the definition of $S^{f}(x)$, it is clear that $S^{f}(x)+\mathcal{L}((x-1,1),(x,f(x)))$ is a tiling of $\mathbb{R}^2$. Define $\Lambda^{f}$ by
\begin{equation}\label{lambdauplattice}
\Lambda^{f}(x)=\mathcal{L}((x-1,1),(x,f(x))).
\end{equation}

Since $K_f\subset S^{f}(x)$, the family $K_f+\Lambda^{f}(x)$ is a packing of $\mathbb{R}^2$ with density $\frac{|K_f|}{|S^{f}(x)|}$. We thus have
$$\delta_L(K_f)\geq \frac{|K_f|}{|S^{f}(x)|}.$$
Therefore
$$\delta_L(K_f)\geq \frac{|K_f|}{A^{f}}.$$
On the other hand, it is obvious that for any translative packing $\{K_f+u_i\}$ with copies of $K_f$, there is at most one $j$ such that $K_f+u_j$ may intersect $(U\setminus K_f)+u_i$, for each $i$, where $U=(0,1)\times (0,1)$. This can be deduced that
$$\delta_T(K_f)\leq \frac{|K_f|}{A^{f}}.$$
We have thus proved

\begin{thm}\label{deltawitharea}
$\delta_L(K_f)=\delta_T(K_f)=\frac{|K_f|}{A^{f}}$.
\end{thm}

By this theorem, $\Lambda^{f}$ can be seen as a mapping from $\mathcal{X}^{f}$ to $\Delta(K_f)$. Furthermore, since $K_f$ is not a square, from the discussion above, one can easily show that $\Lambda^{f}$ is surjective. Now let $x,x'\in \mathcal{X}^{f}$. It is obvious that
$$\Lambda^{f}(x)\cap C_f=\{(x,f(x))\}$$
If $x\neq x'$, then $(x,f(x))\neq (x',f(x'))$. Therefore $\Lambda^{f}(x)\neq\Lambda^{f}(x')$. It follows that $\Lambda^{f}$ is injective. We immediately obtain

\begin{thm}\label{deltabijection}
$\Lambda^{f}$ is a bijection from $\mathcal{X}^{f}$ to $\Delta(K_f)$.
\end{thm}

\section{Determining $\vartheta_L(x,y)$ and $\Theta(x,y)$}\label{thetadetermining}
 We assume that $(x,y)\in  D\setminus\{(1,1)\}$ (hence $K_{x,y}$ is not a unit square), and let

\begin{equation}\label{f}
\bar{f}(x')=
\begin{cases}
1-\frac{(1-y)x'}{x} &0\leq x'< x,\\
\frac{(1-x')y}{1-x}& x\leq x'\leq1,
\end{cases}
\end{equation}
Clearly, we have $K_{x,y}=K_{\bar{f}}$. Denote by $I_1$ and $I_2$ the intervals $[0,x]$ and $[x,1]$, respectively. For $i,j\in\{1,2\}$ and $i\leq j$, we define
$$A_{ij}(x,y)=\max_{x_1\in I_i, x_2\in I_j,x_1\leq x_2}|S_{\bar{f}}(x_1,x_2)|.$$
Let $A_*(x,y)=\max \{A_{12}(x,y), A_{11}(x,y), A_{22}(x,y)\}$. By Theorem \ref{transequallattice_witharea}, we have
\begin{equation}\label{relationship theta and s}
\vartheta_L(x,y)=\frac{|K_{x,y}|}{A_*(x,y)}=\frac{x+y}{2A_*(x,y)}.
\end{equation}
In order to determine $\Theta(x,y)$, we define
$$\tilde{\mathcal{X}}_{ij}(x,y)=\{(x_1,x_2) : x_1\in I_i, x_2\in I_j, x_1\leq x_2, |S_{\bar{f}}(x_1,x_2)|=A_{ij}(x,y)\},$$
and
$$\mathcal{X}_{ij}(x,y)=\{(x_1,x_2) : x_1\in I_i, x_2\in I_j, x_1\leq x_2, |S_{\bar{f}}(x_1,x_2)|=A_*(x,y)\},$$
where $i,j\in\{1,2\}$ and $i\leq j$.
Obviously, we have
\begin{equation}\label{p and p prime}
\mathcal{X}_{ij}(x,y)=
\begin{cases}
\tilde{\mathcal{X}}_{ij}(x,y) & A_{ij}(x,y)=A_*(x,y),\\
\emptyset & otherwise.
\end{cases}
\end{equation}
Let $\mathcal{X}_*(x,y)=\mathcal{X}_{12}(x,y)\cup\mathcal{X}_{11}(x,y)\cup \mathcal{X}_{22}(x,y)$. By Theorem \ref{thetabijection} and (\ref{lambdadownlattice}), we know that
\begin{equation}\label{theta and s}
\Theta(x,y)=\{\mathcal{L}((x_1-x_2,\bar{f}(x_1)),(x_1,\bar{f}(x_2))) : (x_1,x_2)\in \mathcal{X}_*(x,y)\}
\end{equation}

\subsection{$A_{12}(x,y)$ and $\tilde{\mathcal{X}}_{12}(x,y)$}

\begin{figure}[!ht]
  \centering
    \includegraphics[scale=.85]{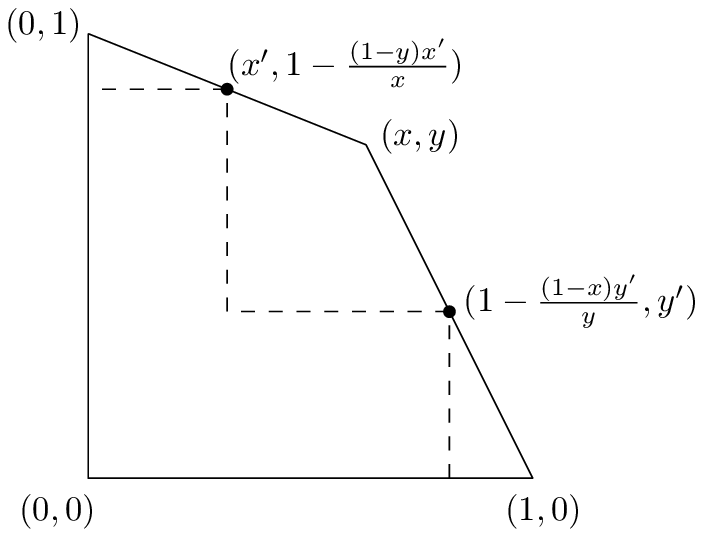}
    \caption{$S_{\bar{f}}\left(x',1-\frac{(1-x)y'}{y}\right)$}\label{a12}
\end{figure}

 Let $G(x',y')$ denote the area of $S_{\bar{f}}(x',1-\frac{(1-x)y'}{y})$, where $0\leq x'\leq x$ and $0\leq y'\leq y$ (see Figure \ref{a12}). Then we get

$$A_{12}(x,y)=\max_{0\leq x'\leq x,0
\leq y'\leq y} G(x',y').$$
By computations, we can obtain
\begin{equation}\label{function G}
G(x',y')=-\frac{(1-y)x'^2}{x}-\frac{(1-x)y'^2}{y}+x'+y'-x'y',
\end{equation}
where $0\leq x'\leq x$ and $0\leq y' \leq y$. Obviously, $G$ is a convex quadratic function with respect to $x'$ (or $y'$).
When $y'$ is fixed, by considering the critical point, it is easy to see that $G(x',y')$ reaches its maximum value at
\begin{equation}\label{x max G}
x'=
\begin{cases}
x & 0\leq y' \leq 2y-1,\\
\frac{x(1-y')}{2(1-y)} & 2y-1< y' \leq y.
\end{cases}
\end{equation}
Here we note that, when $y<1$, we have $\frac{1-y'}{2(1-y)}\leq 1$ is equivalent to $2y-1\leq y'$.
By substituting (\ref{x max G}) into (\ref{function G}), we obtain
\begin{equation}\label{g}
g(y')=
\begin{cases}
-\frac{(1-x)y'^2}{y}+(1-x)y'+xy & 0\leq y' \leq 2y-1,\\
\frac{1}{2}\left(\frac{x}{2(1-y)}-\frac{2(1-x)}{y}\right)y'^2+\left(1-\frac{x}{2(1-y)}\right)y' +\frac{x}{4(1-y)} & 2y-1< y' \leq y,
\end{cases}
\end{equation}
and hence
$$A_{12}(x,y)= \max_{0\leq y'\leq y} g(y').$$
The derivative of $g$ with respect to $y'$ is
\begin{equation}\label{derivative g}
\frac{dg}{dy'}=
\begin{cases}
-\frac{2(1-x)y'}{y}+(1-x) & 0< y'< 2y-1,\\
\left(\frac{x}{2(1-y)}-\frac{2(1-x)}{y}\right)y'+1-\frac{x}{2(1-y)} & 2y-1 < y' < y.
\end{cases}
\end{equation}

\begin{rem} \label{property of g}
Obviously, the function $g$ is a convex quadratic function on $[0,2y-1]$, and if $\frac{y}{2}\leq 2y-1$ (i.e., $y\geq \frac{2}{3}$), then $g$ reaches its maximum value on $[0,2y-1]$ at $y'=\frac{y}{2}$.
\end{rem}

Now we divide $D\setminus\{(1,1)\}$ into six sets
\begin{align*}
D_1=&\{(x,y)\in D : 4(1-x)(1-y)-xy< 0,~x<y\},\\
D_2=&\{(x,y)\in D : y\geq \frac{2}{3},~4(1-x)(1-y)-xy> 0\},\\
D_3=&\{(x,y)\in D : y < \frac{2}{3}\},\\
D_4=&\{(x,y)\in D : 4(1-x)(1-y)-xy=0,~x<y\},\\
D_5=&\{(x,y)\in D : \frac{2}{3}<x=y<1\},\\
D_6=&\{(\frac{2}{3},\frac{2}{3})\}.
\end{align*}

\begin{figure}[!ht]
  \centering
    \includegraphics[scale=1]{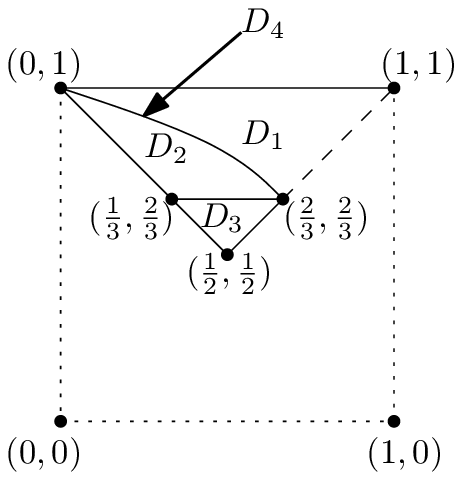}
    \caption{$D_i$}
\end{figure}

\begin{rem}\label{identities}
By computations, we can obtain the following identities
\begin{itemize}
\item[(a)] $4(1-x)(1-y)-xy=3\left(y-\frac{4}{3}\right)\left(x-\frac{4}{3}\right)-\frac{4}{3}$.
\item[(b)] $(2(1-y)-x)-(4(1-x)(1-y)-xy)=(1-y)(3x-2)$.
\item[(c)] $y(2(1-y)-x)-(2y-1)(4(1-x)(1-y)-xy)=2(1-x)(1-y)(2-3y)$.
\end{itemize}
\end{rem}

\begin{enumerate}
\item[\textbf{Case 1}] $(x,y)\in D_1$. Since $4(1-x)(1-y)-xy<0$, $g$ is a concave quadratic function on $[2y-1,y]$. Therefore, the function $g$ reaches its maximum value on  $[2y-1,y]$ at $y'=2y-1$ or $y'=y$. On the other hand, since $y\geq \frac{2}{3}$, from Remark \ref{property of g}, we know that $g$ reaches its maximum value on $[0,2y-1]$ at $y'=\frac{y}{2}$. From (\ref{g}), we obtain
    $$g\left(\frac{y}{2}\right)=\frac{y(1+3x)}{4} \text{~~and~~} g(y)=\frac{x(1+3y)}{4}.$$
    Since $x< y$, we have $g(y)< g(\frac{y}{2})$. Therefore $g$ reaches its maximum value on $[0,y]$ at $y'=\frac{y}{2}$.
    Hence, we know that
    $$A_{12}(x,y)=\frac{y(1+3x)}{4}.$$
    Furthermore, from (\ref{x max G}), we immediately get (as in Figure \ref{a12}, let $x'=x$ and $y'=\frac{y}{2}$)
    $$\tilde{\mathcal{X}}_{12}(x,y)=\left\{\left(x,\frac{1+x}{2}\right)\right\}.$$
\item[\textbf{Case 2}] $(x,y)\in D_2$. Since
$$4(1-x)(1-y)-xy>0 \text{~~and~~} y\geq\frac{2}{3},$$
from Remark \ref{identities} (c), we can easily obtain
\begin{equation}\label{critical point}
\frac{y(2(1-y)-x)}{4(1-x)(1-y)-xy}\leq2y-1.
\end{equation}
On the other hand, $g$ is convex on $[2y-1,y]$, since $4(1-x)(1-y)-xy>0$. It immediately follows from (\ref{derivative g}) and (\ref{critical point}) that $g$ reaches it maximum value on $[2y-1,y]$ at $y'=2y-1$. From Remark \ref{property of g}, we know that $g$ reaches its maximum value on $[0,y]$ at $y'=\frac{y}{2}$. Therefore, we obtain
$$A_{12}(x,y)=\frac{y(1+3x)}{4}$$
and
$$\tilde{\mathcal{X}}_{12}(x,y)=\left\{\left(x,\frac{1+x}{2}\right)\right\}.$$
\item[\textbf{Case 3}] $(x,y)\in D_3$. In this case, one can see that $4(1-x)(1-y)-xy>0$ and $x<\frac{2}{3}$. From Remark \ref{identities} (b), we have
    \begin{equation}\label{in case 3}
    \frac{2(1-y)-x}{4(1-x)(1-y)-xy}<1.
    \end{equation}
    Since $y<\frac{2}{3}$, it follows from Remark \ref{identities} (c) and (\ref{in case 3}) that
    $$2y-1<\frac{y(2(1-y)-x)}{4(1-x)(1-y)-xy}<y.$$
    From (\ref{derivative g}), we know that $g$ reaches its maximum value on $[2y-1,y]$ at
    $$y'=\frac{y(2(1-y)-x)}{4(1-x)(1-y)-xy}.$$
    On the other hand, we have $2y-1<\frac{y}{2}$, since $y<\frac{2}{3}$. Hence, $g$ reaches its maximum value on $[0,2y-1]$ at $y'=2y-1$. This immediately implies that $g$ reaches its maximum value at
    $$y'=\frac{y(2(1-y)-x)}{4(1-x)(1-y)-xy}.$$
    Hence, we have
    $$A_{12}(x,y)=g\left(\frac{y(2(1-y)-x)}{4(1-x)(1-y)-xy}\right)=\frac{x(1-x)+y(1-y)-xy}{4(1-x)(1-y)-xy},$$
    and
    \begin{equation*}
    \begin{split}
    \tilde{\mathcal{X}}_{12}(x,y)=\left\{\left(\frac{x(2(1-x)-y)}{4(1-x)(1-y)-xy},\frac{(2-x)(1-y)-x^2}{4(1-x)(1-y)-xy}\right)\right\}
    \end{split}
    \end{equation*}
\item[\textbf{Case 4}] $(x,y)\in D_4$. In this case, we have $4(1-x)(1-y)-xy=0$, and $x<y$. Hence $g$ is a non-constant linear function on $[2y-1,y]$. Therefore, $g$ reaches its maximum value on $[2y-1,y]$ at $y'=2y-1$ or $y$. By the same argument as case 1, we have
     $$A_{12}(x,y)=\frac{y(1+3x)}{4}.$$
    and
    $$\tilde{\mathcal{X}}_{12}(x,y)=\left\{\left(x,\frac{1+x}{2}\right)\right\}.$$
\item[\textbf{Case 5}] $(x,y)\in D_5$, i.e., $\frac{2}{3}<x=y<1$. Since $4(1-x)(1-y)-xy<0$, $g$ is a concave quadratic function on $[2y-1,y]$. Therefore, the function $g$ reaches its maximum value on $[2y-1,y]$ at $y'=2y-1$ or $y'=y$. On the other hand, since $y>\frac{2}{3}$, the function $g$ reaches its maximum value on $[0,2y-1]$ at $y'=\frac{y}{2}$. Note that
    $$g(y)=\frac{x(1+3y)}{4}=\frac{y(1+3x)}{4}=g\left(\frac{y}{2}\right),$$
    since $x=y$. Therefore $g$ reaches its maximum value on $[0,y]$ at $y'=\frac{y}{2}$ and $y'=y$. Hence we obtain
    $$A_{12}(y,y)=\frac{y(1+3y)}{4}$$
    and
    $$\tilde{\mathcal{X}}_{12}(y,y)=\left\{\left(y,\frac{1+y}{2}\right),\left(\frac{y}{2},y\right)\right\}.$$
\item[\textbf{Case 6}] $(x,y)\in D_6,$ i.e.,$x=y=\frac{2}{3}$. From (\ref{g}), we have
    \begin{equation*}
    g(y')=
    \begin{cases}
    -\frac{y'^2}{2}+\frac{y'}{3}+\frac{4}{9} & 0\leq y' \leq \frac{1}{3},\\
     \frac{1}{2}& \frac{1}{3}\leq y' \leq \frac{2}{3}.
    \end{cases}
    \end{equation*}
    Hence we get
    $$A_{12}\left(\frac{2}{3},\frac{2}{3}\right)=\frac{1}{2}$$
    and
    $$\tilde{\mathcal{X}}_{12}\left(\frac{2}{3},\frac{2}{3}\right)
    =\left\{\left(1-y',1-\frac{y'}{2}\right) : \frac{1}{3}\leq y'\leq\frac{2}{3}\right\}$$
\end{enumerate}

In summary, we have
\begin{equation}\label{s1}
A_{12}(x,y)=
\begin{cases}
\frac{y(1+3x)}{4} & y\geq \frac{2}{3},\\
\frac{x(1-x)+y(1-y)-xy}{4(1-x)(1-y)-xy} & y<\frac{2}{3},
\end{cases}
\end{equation}
and
\begin{equation} \label{p prime 1}
\tilde{\mathcal{X}}_{12}(x,y)=
\begin{cases}
\left\{\left(x,\frac{1+x}{2}\right)\right\} & \frac{2}{3}\leq y,x<y,\\
\left\{\left(\frac{x(2(1-x)-y)}{4(1-x)(1-y)-xy},\frac{(2-x)(1-y)-x^2}{4(1-x)(1-y)-xy}\right)\right\} & y<\frac{2}{3},\\
\left\{\left(y,\frac{1+y}{2}\right),\left(\frac{y}{2},y\right)\right\}
& \frac{2}{3}<x=y< 1,\\
\left\{\left(1-t,1-\frac{t}{2}\right) : \frac{1}{3}\leq t\leq\frac{2}{3}\right\} & x=y=\frac{2}{3}.
\end{cases}
\end{equation}
where $(x,y)\in D\setminus\{(1,1)\}$.

\subsection{$A_{22}(x,y)$ and $\tilde{\mathcal{X}}_{22}(x,y)$}
\begin{figure}[!ht]
  \centering
    \includegraphics[scale=1]{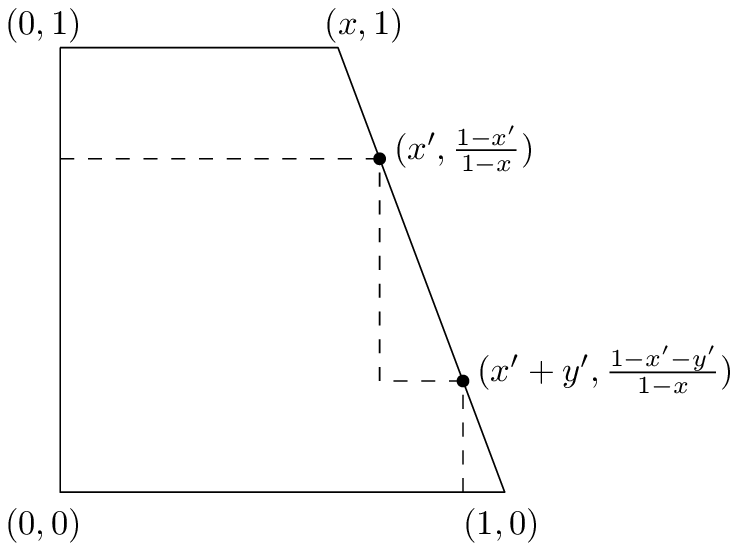}
    \caption{$S_{\bar{f}}\left(x',x'+y'\right)$}\label{a22}
\end{figure}

We first determine the case $y=1$ (and $x\neq 1$). Let $G_2(x',y')$ be the area of $S_{\bar{f}}(x',x'+y')$ (as in Figure \ref{a22}), where $x\leq x'\leq 1$ and $0\leq y'\leq 1-x'$, then
$$
G_2(x',y')= \frac{1-x'^2}{2(1-x)}-\frac{1}{2(1-x)}\left(y'^2+(1-x'-y')^2\right).
$$
When $x'$ is fixed, by Arithmetic Mean-Quadratic Mean inequality, we know that $G_2(x',y')$ reaches its maximum value at
\begin{equation}\label{y max G 2}
y'=\frac{1-x'}{2}.
\end{equation}

Let
$$
g_2(x')=G_2\left(x',\frac{1-x'}{2}\right)=\frac{(1+3x')(1-x')}{4(1-x)}.
$$
Obviously, $g_2$ is a convex quadratic function and it reaches its maximum value on $[x,1]$ at
\begin{equation}\label{x max G 2}
x'=
\begin{cases}
\frac{1}{3} & 0\leq x\leq \frac{1}{3},\\
x & \frac{1}{3}\leq x<1.
\end{cases}
\end{equation}
and hence
$$
A_{22}(x,1)=
\begin{cases}
\frac{1}{3(1-x)} & 0 \leq x\leq \frac{1}{3},\\
\frac{1+3x}{4} & \frac{1}{3}\leq x<1.
\end{cases}
$$
Furthermore, From (\ref{y max G 2}) and (\ref{x max G 2}), we obtain (see Figure \ref{a22})
$$\tilde{\mathcal{X}}_{22}(x,1)=
\begin{cases}
\left\{\left(\frac{1}{3},\frac{2}{3}\right)\right\} &0 \leq x\leq \frac{1}{3},\\
\left\{\left(x,\frac{1+x}{2}\right)\right\} & \frac{1}{3}\leq x<1.
\end{cases}
$$

\begin{figure}[!ht]
  \centering
    \includegraphics[scale=.8]{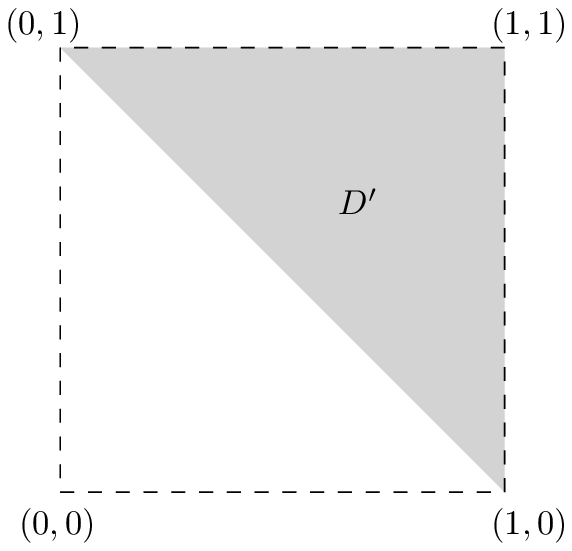}
    \caption{$D'$}
\end{figure}

In general case, we denote $D'=\{(x,y) : 0\leq x \leq 1,~0\leq y \leq 1,~x+y\geq 1\}$. Let $(x,y)\in D'$ and $y\neq 0$. By using affine transformation ($\tau :\mathbb{R}^2\rightarrow\mathbb{R}^2, \tau(s,t)= (s,\frac{t}{y})$), one can show that
$$A_{22}(x,y)=yA_{22}(x,1),$$
and $\tilde{\mathcal{X}}_{22}(x,y)=\tilde{\mathcal{X}}_{22}(x,1)$.
Therefore,
\begin{equation}\label{s2}
A_{22}(x,y)=
\begin{cases}
\frac{y}{3(1-x)} & x\leq \frac{1}{3},\\
\frac{y(1+3x)}{4} & x\geq \frac{1}{3}.
\end{cases}
\end{equation}
and
\begin{equation}\label{p prime 2}
\tilde{\mathcal{X}}_{22}(x,y)=
\begin{cases}
\left\{\left(\frac{1}{3},\frac{2}{3}\right)\right\} & x\leq \frac{1}{3},\\
\left\{\left(x,\frac{1+x}{2}\right)\right\} & x\geq \frac{1}{3},
\end{cases}
\end{equation}
where $(x,y)\in D'\setminus\{(1,y') : 0\leq y'\leq 1\}$.

\begin{figure}[!ht]
  \centering
    \includegraphics[scale=1]{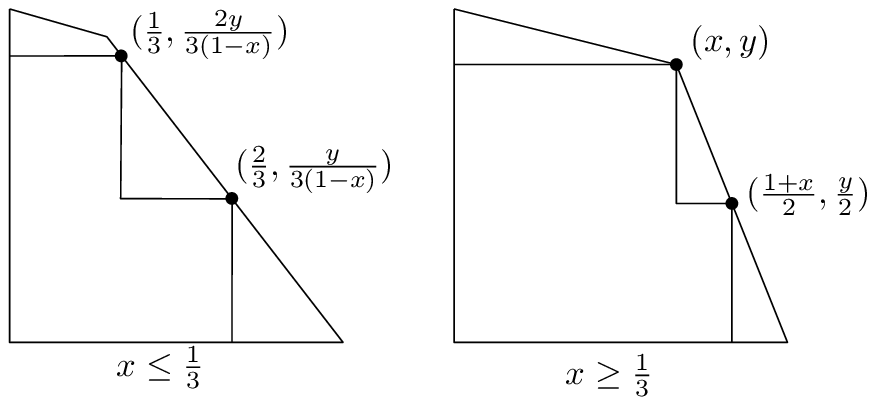}
    \caption{$\tilde{\mathcal{X}}_{22}(x,y)$}\label{x22}
\end{figure}

\subsection{$A_{11}(x,y)$ and $\tilde{\mathcal{X}}_{11}(x,y)$}
When $y=1$, one can easily obtain
$$
A_{11}(x,1)=x,
$$
and
$$\tilde{\mathcal{X}}_{11}(x,1)=\{(x',x) : 0\leq x'\leq x\}.$$
When $y\neq 1$. From (\ref{s2}), (\ref{p prime 2}) and Figure \ref{x22}, by swapping $x$ with $y$, it is easy to see that
\begin{equation*}
A_{11}(x,y)=
\begin{cases}
\frac{x}{3(1-y)} & 0\leq y\leq \frac{1}{3},\\
\frac{x(1+3y)}{4} & \frac{1}{3}\leq y<1,
\end{cases}
\end{equation*}
and
\begin{equation*}
\tilde{\mathcal{X}}_{11}(x,y)=
\begin{cases}
\left\{\left(\frac{x}{3(1-y)},\frac{2x}{3(1-y)}\right)\right\} & 0\leq y\leq \frac{1}{3},\\
\left\{\left(\frac{x}{2},x\right)\right\} & \frac{1}{3}\leq y<1.
\end{cases}
\end{equation*}
Therefore, for $(x,y)\in D'$, we have
\begin{equation}\label{s3}
A_{11}(x,y)=
\begin{cases}
\frac{x}{3(1-y)} & y\leq \frac{1}{3},\\
\frac{x(1+3y)}{4} & y\geq \frac{1}{3},
\end{cases}
\end{equation}
and
\begin{equation}\label{p prime 3}
\tilde{\mathcal{X}}_{11}(x,y)=
\begin{cases}
\left\{\left(\frac{x}{3(1-y)},\frac{2x}{3(1-y)}\right)\right\} & 0\leq y\leq \frac{1}{3},\\
\left\{\left(\frac{x}{2},x\right)\right\} & \frac{1}{3}\leq y<1,\\
\{(t,x) : 0\leq t\leq x\} & y=1¡£
\end{cases}
\end{equation}

\subsection{Determining $A_*(x,y)$}\label{Adownsubsec}
Recall that $A_*(x,y)=\max\{A_{12}(x,y),A_{11}(x,y),A_{22}(x,y)\}$. We divide $D$ into three sets
\begin{align*}
B_1= &\{(x,y)\in D : x\leq \frac{1}{3}\},\\
B_2= &\{(x,y)\in D : \frac{1}{3}< x\leq1,~y\geq \frac{2}{3}\},\\
B_3= &\{(x,y)\in D : y< \frac{2}{3}\}.
\end{align*}

\begin{figure}[!ht]
  \centering
    \includegraphics[scale=1]{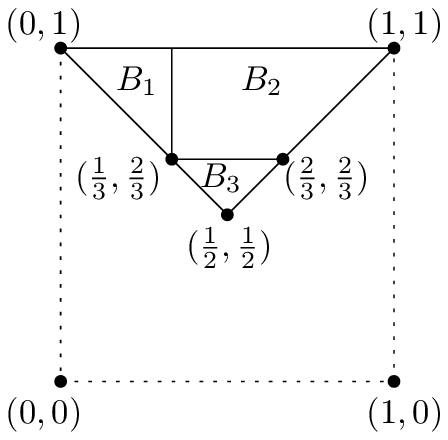}
    \caption{$B_1,B_2$ and $B_3$}
\end{figure}
\begin{enumerate}
\item[\textbf{Case 1}] $(x,y)\in B_1$. We have
$$A_*(x,y)=\max \left\{\frac{y(1+3x)}{4},\frac{x(1+3y)}{4},\frac{y}{3(1-x)}\right\}$$
Since $0\leq x< y\leq 1$, one can show that
\begin{equation}\label{case1}
\frac{x(1+3y)}{4}< \frac{y(1+3x)}{4}\leq \frac{y}{3(1-x)},
\end{equation}
with $\frac{y(1+3x)}{4}= \frac{y}{3(1-x)}$ only when $x=\frac{1}{3}$. Therefore
\begin{equation}\label{s case 1}
A_*(x,y)=\frac{y}{3(1-x)}.
\end{equation}

\item[\textbf{Case 2}] $(x,y)\in B_2$. Note that in this case, we have
$$A_{12}(x,y)=A_{22}(x,y)=\frac{y(1+3x)}{4},$$
thus
\begin{equation}\label{s case 2}
A_*(x,y)=\max \left\{\frac{y(1+3x)}{4},\frac{x(1+3y)}{4}\right\}=\frac{y(1+3x)}{4}.
\end{equation}

\item[\textbf{Case 3}] $(x,y)\in B_3$. We have
$$A_*(x,y)=\max \left\{\frac{x(1-x)+y(1-y)-xy}{4(1-x)(1-y)-xy},\frac{x(1+3y)}{4},\frac{y(1+3x)}{4}\right\}.
$$
Note that
$$4(x(1-x)+y(1-y)-xy)-y(1+3x)(4(1-x)(1-y)-xy)=x(1-x)(3y-2)^2,$$
we thus have
\begin{equation}\label{inequality case 3}
\frac{x(1-x)+y(1-y)-xy}{4(1-x)(1-y)-xy}\geq \frac{y(1+3x)}{4},
\end{equation}
with equality only when $y=\frac{2}{3}$.
Hence
\begin{equation}\label{s case 3}
A_*(x,y)=\frac{x(1-x)+y(1-y)-xy}{4(1-x)(1-y)-xy}
\end{equation}
\end{enumerate}

By combining (\ref{s case 1}), (\ref{s case 2}), (\ref{s case 3}) and (\ref{relationship theta and s}), we have thus proved the second part of Theorem \ref{main theorem}.

\subsection{Determining $\mathcal{X}_*(x,y)$ and $\Theta(x,y)$}
To determine $\mathcal{X}_*(x,y)$, we divide $D\setminus\{(1,1)\}$ into five sets
\begin{align*}
E_1=&\{(x,y)\in D : x\leq \frac{1}{3}\},\\
E_2=&\{(x,y)\in D : \frac{1}{3}< x<y,~\frac{2}{3}\leq y\},\\
E_3=&\{(x,y)\in D : y < \frac{2}{3}\},\\
E_4=&\{(x,y)\in D : \frac{2}{3}<x=y<1\},\\
E_5=&\left\{\left(\frac{2}{3},\frac{2}{3}\right)\right\}.
\end{align*}

\begin{enumerate}
\item[\textbf{Case 1}] $(x,y)\in E_1$. From (\ref{case1}) and (\ref{s case 1}), we know that
$$A_{11}(x,y)<A_{12}(x,y)\leq A_{22}(x,y)=A_*(x,y),$$
Moreover, we have
$$A_{12}(x,y)= A_{22}(x,y)$$
only when $x=\frac{1}{3}$. From (\ref{p prime 1}),(\ref{p prime 2}) and (\ref{p and p prime}), one can obtain
\begin{align*}
\mathcal{X}_{12}(x,y)&=
\begin{cases}
\tilde{\mathcal{X}}_{12}(x,y) &x=\frac{1}{3},\\
\emptyset & otherwise,
\end{cases}
\\
&=
\begin{cases}
\left\{\left(\frac{1}{3},\frac{2}{3}\right)\right\} &x=\frac{1}{3},\\
\emptyset & otherwise,
\end{cases}
\end{align*}
$$\mathcal{X}_{22}(x,y)=\tilde{\mathcal{X}}_{22}(x,y)=\left\{\left(\frac{1}{3},\frac{2}{3}\right)\right\}$$
and
$$\mathcal{X}_{11}(x,y)=\emptyset.$$
Hence
\begin{equation}\label{p case 1}
\mathcal{X}_*(x,y)
=\left\{\left(\frac{1}{3},\frac{2}{3}\right)\right\}
\end{equation}
\item[\textbf{Case 2}] $(x,y)\in E_2$. From (\ref{s1}), (\ref{s2}), (\ref{s3}) and (\ref{s case 2}), one can see that
    $$A_{12}(x,y)=A_{22}(x,y)=A_*(x,y)$$
    and
    $$A_{11}(x,y)<A_*(x,y).$$
   From (\ref{p prime 1}),(\ref{p prime 2}) and (\ref{p and p prime}), we have
    $$\mathcal{X}_{12}(x,y)=\tilde{\mathcal{X}}_{12}(x,y)=\left\{\left(x,\frac{1+x}{2}\right)\right\},$$
    $$\mathcal{X}_{22}(x,y)=\tilde{\mathcal{X}}_{22}(x,y)=\left\{\left(x,\frac{1+x}{2}\right)\right\},$$
    and
    $$\mathcal{X}_{11}(x,y)=\emptyset.$$
    Thus,
    \begin{equation}\label{p case 2}
    \mathcal{X}_*(x,y)=\left\{\left(x,\frac{1+x}{2}\right)\right\}
    \end{equation}
\item[\textbf{Case 3}] $(x,y)\in E_3$. From (\ref{s1}), (\ref{s2}), (\ref{s3}), (\ref{inequality case 3}) and (\ref{s case 3}), we obtain
    $$A_{11}(x,y)\leq A_{22}(x,y)<A_{12}(x,y)=A_*(x,y).$$
    From (\ref{p prime 1}),(\ref{p prime 2}) and (\ref{p and p prime}), we get
    \begin{align*}
    \mathcal{X}_{12}(x,y)&=\tilde{\mathcal{X}}_{12}(x,y)\\
    &=\left\{\left(\frac{x(2(1-x)-y)}{4(1-x)(1-y)-xy},\frac{(2-x)(1-y)-x^2}{4(1-x)(1-y)-xy}\right)\right\} \end{align*}
    and
    $$\mathcal{X}_{22}(x,y)=\mathcal{X}_{11}(x,y)=\emptyset.$$
    Therefore
    \begin{equation}\label{p case 3}
    \begin{split}
    \mathcal{X}_*(x,y)=\left\{\left(\frac{x(2(1-x)-y)}{4(1-x)(1-y)-xy},\frac{(2-x)(1-y)-x^2}{4(1-x)(1-y)-xy}\right)\right\}  \end{split}
    \end{equation}
\item[\textbf{Case 4}] $(x,y)\in E_4$.  From (\ref{s1}), (\ref{s2}), (\ref{s3}) and (\ref{s case 2}), since $x=y$, we have
    $$A_{12}(y,y)=A_{22}(y,y)=A_{11}(y,y)=A_*(y,y).$$
    From (\ref{p prime 1}),(\ref{p prime 2}), (\ref{p prime 3}) and (\ref{p and p prime}), we obtain
    $$\mathcal{X}_{12}(y,y)=\tilde{\mathcal{X}}_{12}(y,y)=\left\{\left(y,\frac{1+y}{2}\right),\left(\frac{y}{2},y\right)\right\},$$
    $$\mathcal{X}_{22}(y,y)=\tilde{\mathcal{X}}_{22}(y,y)=\left\{\left(y,\frac{1+y}{2}\right)\right\},$$
    and
    $$\mathcal{X}_{11}(y,y)=\tilde{\mathcal{X}}_{11}(y,y)=\left\{\left(\frac{y}{2},y\right)\right\}.$$
    Hence
    \begin{equation}\label{p case 4}
    \mathcal{X}_*(y,y)=\left\{\left(y,\frac{1+y}{2}\right),\left(\frac{y}{2},y\right)\right\}.
    \end{equation}
\item[\textbf{Case 5}] $(x,y)\in E_5$, i.e., $x=y=\frac{2}{3}$. From (\ref{s case 1}), (\ref{s case 2}) and (\ref{s case 3}), we have
    $$A_{12}\left(\frac{2}{3},\frac{2}{3}\right)=A_{22}\left(\frac{2}{3},\frac{2}{3}\right)=A_{11}\left(\frac{2}{3},\frac{2}{3}\right)=\frac{1}{2}.$$
     From (\ref{p prime 1}),(\ref{p prime 2}), (\ref{p prime 3}) and (\ref{p and p prime}), we know that
     $$\mathcal{X}_{12}\left(\frac{2}{3},\frac{2}{3}\right)=\tilde{\mathcal{X}}_{12}\left(\frac{2}{3},\frac{2}{3}\right)=\left\{\left(1-t,1-\frac{t}{2}\right) : \frac{1}{3}\leq t\leq\frac{2}{3}\right\},$$
      $$\mathcal{X}_{22}\left(\frac{2}{3},\frac{2}{3}\right)=\tilde{\mathcal{X}}_{22}(x,y)=\left\{\left(\frac{2}{3},\frac{5}{6}\right)\right\},$$
    and
    $$\mathcal{X}_{11}\left(\frac{2}{3},\frac{2}{3}\right)=\tilde{\mathcal{X}}_{11}\left(\frac{2}{3},\frac{2}{3}\right)=\left\{\left(\frac{1}{3},\frac{2}{3}\right)\right\}.$$
    Hence
    \begin{equation}\label{p case 5}
    \mathcal{X}_*\left(\frac{2}{3},\frac{2}{3}\right)=\left\{\left(1-t,1-\frac{t}{2}\right) : \frac{1}{3}\leq t\leq\frac{2}{3}\right\}.
    \end{equation}
\end{enumerate}
By combining (\ref{f}), (\ref{theta and s}), (\ref{p case 1}), (\ref{p case 2}), (\ref{p case 3}), (\ref{p case 4}), and (\ref{p case 5}), we immediately obtain the second part of Theorem \ref{main theorem_1}. Here, we note that for the case $(x,y)=(1,1)$, the result is trivial.

\section{Determining $\delta_{L}(x,y)$ and $\Delta(x,y)$}
Assume that $(x,y)\neq(1,1)$. Let $\bar{f}$ be the function defined in Section \ref{thetadetermining} (see (\ref{f})). Recall that $I_1$ and $I_2$ denote the intervals $[0,x]$ and $[x,1]$, respectively. For $i=1,2$,
we define
$$A^i(x,y)=\min_{x'\in I_i}|S^{\bar{f}}(x')|.$$
Let $A^*(x,y)=\min\{A^1(x,y),A^2(x,y)\}$.

In order to determine $\Delta(x,y)$, we define
$$\tilde{\mathcal{X}}^i(x,y)=\{x' : x'\in I_i,~|S^{\bar{f}}(x)|=A^i(x,y)\},$$
and
$$\mathcal{X}^i(x,y)=\{x' : x'\in I_i,~|S^{\bar{f}}(x')|=A^*(x,y)\},$$
where $i=1,2$.
Obviously, for $i=1,2$, we have
\begin{equation}\label{p star and p prime}
\mathcal{X}^i(x,y)=
\begin{cases}
\tilde{\mathcal{X}}^i(x,y) & A^i(x,y)=A^*(x,y),\\
\emptyset & otherwise.
\end{cases}
\end{equation}
Let $\mathcal{X}^*(x,y)=\mathcal{X}^1(x,y)\cup\mathcal{X}^2(x,y)$.
By Theorem \ref{deltabijection} and (\ref{lambdauplattice}), we know that
\begin{equation}\label{delta and s}
\Delta(x,y)=\{\mathcal{L}((x'-1,1),(x',\bar{f}(x'))) : x'\in\mathcal{X}^*(x,y)\}.
\end{equation}

\subsection{$A^*(x,y)$}
Suppose that $(x,y)\in D'\setminus\{(1,1)\}$. When $y=1$, we can easily get
$$A^1(x,1)=1,$$
and
$$\tilde{\mathcal{X}}^1(x,1)=I_1.$$
We now suppose that $y\neq 1$. Given any $x'\in I_1$. By elementary computations (see Figure \ref{sfup1}), we have
$$|S^{\bar{f}}(x')|=1-\frac{(1-y)(1-x')x'}{x}.$$
Obviously, it is a concave quadratic function with respect to $x'$. Therefore, $|S^{\bar{f}}(x')|$ reaches its minimum value on $I_1$ at
$$
x'=
\begin{cases}
\frac{1}{2} & x\geq \frac{1}{2},\\
x & x\leq \frac{1}{2}.
\end{cases}
$$
Hence, we obtain
\begin{equation}\label{s 1 star}
A^1(x,y)=
\begin{cases}
1-\frac{1-y}{4x} & x\geq \frac{1}{2},\\
1-(1-x)(1-y) & x\leq \frac{1}{2},
\end{cases}
\end{equation}
and
\begin{equation}\label{p star 1}
\tilde{\mathcal{X}}^1(x,y)=
\begin{cases}
I_1 & y=1,\\
\left\{\frac{1}{2}\right\} & x\geq\frac{1}{2}, y<1,\\
\left\{x\right\} & x\leq \frac{1}{2}, y<1.
\end{cases}
\end{equation}
where $(x,y)\in D'\setminus\{(1,1)\}$.
By symmetry, we immediately get
\begin{equation}\label{s 2 star}
A^2(x,y)=
\begin{cases}
1-\frac{1-x}{4y} & y\geq \frac{1}{2},\\
1-(1-x)(1-y) & y\leq \frac{1}{2},
\end{cases}
\end{equation}
and
\begin{equation}\label{p star 2}
\tilde{\mathcal{X}}^2(x,y)=
\begin{cases}
I_2 & x=1,\\
\left\{1-\frac{1-x}{2y}\right\} & y\geq\frac{1}{2}, x<1,\\
\left\{x\right\} & y\leq \frac{1}{2}, x<1.
\end{cases}
\end{equation}
where $(x,y)\in D'\setminus\{(1,1)\}$.

\begin{figure}[!ht]
  \centering
    \includegraphics[scale=1]{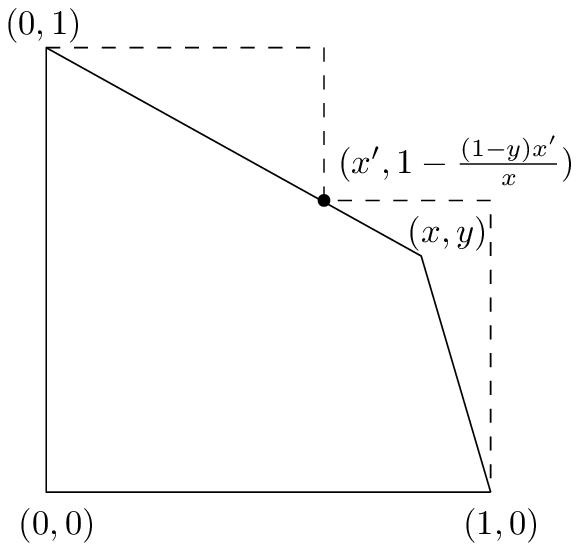}
    \caption{$S^{\bar{f}}(x'),x'\in I_1$}\label{sfup1}
\end{figure}

Now we suppose that $(x,y)\in D\setminus\{(1,1)\}$. Since $0\leq x\leq y\leq 1$ and $x+y\geq 1$, we have
$$y(1-y)\leq x(1-x),$$
and hence
\begin{equation}\label{s star inequality 1}
1-\frac{1-x}{4y}\leq 1-\frac{1-y}{4x},
\end{equation}
with equality only when $x=y$ or $x+y=1$. On the other hand, by Arithmetic Mean-Geometric Mean inequality, we know that
$$y(1-y)\leq \frac{1}{4},$$
therefore
\begin{equation}\label{s star inequality 2}
1-\frac{1-x}{4y}\leq 1-(1-x)(1-y),
\end{equation}
with equality only when $x=1$ or $y=\frac{1}{2}$. From (\ref{s 1 star}),(\ref{s 2 star}),(\ref{s star inequality 1}) and (\ref{s star inequality 2}), we have
\begin{equation}\label{s star}
\begin{split}
A^*(x,y) &=\min \{A^1(x,y),A^2(x,y)\}\\
&=\min \left\{A^1(x,y),1-\frac{1-x}{4y}\right\}\\
&=1-\frac{1-x}{4y},
\end{split}
\end{equation}
where $(x,y)\in D$.
The first part of Theorem \ref{main theorem} immediately follows from Theorem \ref{deltawitharea} and (\ref{s star}).

\subsection{$\Delta(x,y)$}

We divide $D\setminus\{(1,1)\}$ into four sets
\begin{align*}
F_1&=\{(x,y)\in D : \frac{1}{2}<x=y<1\},\\
F_2&=\{(x,y)\in D : x+y=1,y>\frac{1}{2}\},\\
F_3&=\left\{\left(\frac{1}{2},\frac{1}{2}\right)\right\},\\
F_4&=\{(x,y)\in D : x<y, 1<x+y\}.
\end{align*}

\begin{enumerate}
\item[\textbf{Case 1}] $(x,y)\in F_1$. From (\ref{s 1 star}) and (\ref{s 2 star}), since $x=y$, we have
    $$A^1(y,y)=A^2(y,y)=A^*(y,y).$$
    Therefore
    $$\mathcal{X}^1(y,y)=\tilde{\mathcal{X}}^1(y,y)=\left\{\frac{1}{2}\right\},$$
    and
    $$\mathcal{X}^2(y,y)=\tilde{\mathcal{X}}^2(y,y)=\left\{1-\frac{1-y}{2y}\right\}.$$
    Hence
    \begin{equation}\label{p star case 2}
    \mathcal{X}^*(y,y)=\left\{\frac{1}{2},\frac{3}{2}-\frac{1}{2y}\right\}.
    \end{equation}
\item[\textbf{Case 2}] $(x,y)\in F_2$, i.e., $x+y=1$ and $y<\frac{1}{2}$. From (\ref{s 1 star}), (\ref{s 2 star}) and (\ref{s star inequality 2}), we know that
$$A^*(x,y)=A^2(x,y)=1-\frac{1-x}{4y}<1-(1-x)(1-y)=A^1(x,y).$$
Therefore
$$\mathcal{X}^1(x,y)=\emptyset,$$
and
$$\mathcal{X}^2(x,y)=\tilde{\mathcal{X}}^2(x,y)=\left\{1-\frac{1-x}{2y}\right\}=\left\{\frac{1}{2}\right\}.$$
Thus
\begin{equation}\label{p star case 3}
\mathcal{X}^*(x,y)=\left\{\frac{1}{2}\right\}.
\end{equation}
\item[\textbf{Case 3}] $(x,y)\in F_3$, i.e., $x=y=\frac{1}{2}$. One can see that
$$\mathcal{X}^1\left(\frac{1}{2},\frac{1}{2}\right)=\mathcal{X}^2\left(\frac{1}{2},\frac{1}{2}\right)=\left\{\frac{1}{2}\right\}.$$
Thus
\begin{equation}\label{p star case 4}
\mathcal{X}^*\left(\frac{1}{2},\frac{1}{2}\right)=\left\{\frac{1}{2}\right\}.
\end{equation}
\item[\textbf{Case 4}] $(x,y)\in F_4$. In this case, it is clear that
$$
1-\frac{1-x}{4y}< 1-\frac{1-y}{4x},
$$
and
$$
1-\frac{1-x}{4y}< 1-(1-x)(1-y).
$$
Hence, we have
$$A^*(x,y)=A^2(x,y)=1-\frac{1-x}{4y}<A^1(x,y).$$
Thus
$$\mathcal{X}^1(x,y)=\emptyset,$$
and
$$\mathcal{X}^2(x,y)=\tilde{\mathcal{X}}^2(x,y)=\left\{1-\frac{1-x}{2y}\right\}.$$
We obtain
\begin{equation}\label{p star case 5}
\mathcal{X}^*(x,y)=\left\{1-\frac{1-x}{2y}\right\}.
\end{equation}
\end{enumerate}

By combining (\ref{f}), (\ref{delta and s}), (\ref{p star case 2}), (\ref{p star case 3}), (\ref{p star case 4}) and (\ref{p star case 5}), the first part of Theorem \ref{main theorem_1} is complete.

\section{Proof of $\delta_{L}(Q)\vartheta_{L}(Q)\geq 1$ and $\delta_{L}(Q)+\vartheta_{L}(Q)\geq 2$}
To prove $\delta_{L}(Q)\vartheta_{L}(Q)\geq 1$, it suffices to show that $\delta_{L}(x,y)\vartheta_{L}(x,y)\geq 1$, for all $(x,y)\in D$. Let $B_1$, $B_2$ and $B_3$ be the sets defined in Section \ref{Adownsubsec}. Suppose that $(x,y)\in D$.
\begin{enumerate}
\item[\textbf{Case 1}] $(x,y)\in B_1$. Since $x+y\geq 1$, we have
$$x+y+1\geq 2 \geq\frac{4}{3(1-x)},$$
and hence
$$(x+y)^2-1\geq \frac{4(x+y-1)}{3(1-x)}.$$
Therefore
$$3(1-x)((x+y)^2-1)\geq 4(x+y-1).$$
This implies that
$$3(1-x)(x+y)^2\geq 4y-(1-x),$$
thus
$$\delta_{L}(x,y)\vartheta_{L}(x,y)=\frac{3(1-x)(x+y)^2}{4y-(1-x)}\geq 1.$$
\item[\textbf{Case 2}] $(x,y)\in B_2$. Since
$$(2y-x-1)^2\geq 0,$$
we have
$$4y^2-4xy+x^2\geq 4y-2x-1.$$
Hence
\begin{equation*}
\begin{split}
4(x+y)^2 &= 4y^2+8xy+4x^2\\
&\geq 4y-2x-1+12xy+3x^2\\
&=(4y+x-1)(1+3x).
\end{split}
\end{equation*}
Therefore
$$\delta_{L}(x,y)\vartheta_{L}(x,y)=\frac{4(x+y)^2}{(4y-(1-x))(1+3x)}\geq 1.$$
\item[\textbf{Case 3}] $(x,y)\in B_3$. In this case, we have
$$\delta_{L}(x,y)\vartheta_{L}(x,y)=\frac{y(x+y)^2(4(1-x)(1-y)-xy)}{(4y-(1-x))(x(1-x)+y(1-y)-xy)}.$$
To show that $\delta_{L}(x,y)\vartheta_{L}(x,y)\geq 1$, we may define
$$\pi (x,y)=y(x+y)^2(4(1-x)(1-y)-xy)-(4y-(1-x))(x(1-x)+y(1-y)-xy).$$
By computations, we obtain
$$\pi (x,y)=(x+y-1)\pi^*(x,y),$$
where
$$\pi^*(x,y)=(3y^2-4y+1)x^2+(3y^3-5y^2+4y-1)x-(4y^3-4y^2+y).$$
We claim that $\pi^*(x,y)\geq 0$. In fact, since $\frac{1}{2}\leq y<\frac{2}{3}$, we have $3y^2-4y+1=(3y-1)(y-1)<0$. Hence, $\pi^*$ is a convex quadratic function with respect to $x$. Therefore, it suffices to show that $\pi^*(x,y)\geq 0$, when $x+y=1$ or $x=y$. This is obvious, because if $x+y=1$, then
$$\pi^*(x,y)=\pi^*(1-y,y)=-6y^3+7y^2-2y=-y(3y-2)(2y-1)\geq 0,$$
and if $x=y$, then
$$\pi^*(x,y)=\pi^*(y,y)=y(y-1)(3y-2)(2y-1)\geq 0.$$
\end{enumerate}
This completes the proof of the first part of Corollary \ref{main cor}. Moreover, by Arithmetic Mean-Geometric Mean inequality, it immediately follows that
$$\delta_{L}(Q)+\vartheta_{L}(Q)\geq 2\sqrt{\delta_{L}(Q)\vartheta_{L}(Q)}\geq 2.$$

\section{The proof of $\frac{1}{\delta_{L}(Q)}+\frac{1}{\vartheta_{L}(Q)}\geq 2$ and $\vartheta_{L}(Q)\leq 1+\frac{5}{4}\sqrt{1-\delta_{L}(Q)}$}
To prove $\frac{1}{\delta_{L}(Q)}+\frac{1}{\vartheta_{L}(Q)}\geq 2$, it suffices to show that $\frac{1}{\delta_{L}(x,y)}+\frac{1}{\vartheta_{L}(x,y)}\geq 2$, for all $(x,y)\in D$. Suppose that $(x,y)\in D$.

\begin{enumerate}
\item[\textbf{Case 1}] $(x,y)\in B_1$. We have to show that
$$\frac{4y+x-1}{2y(x+y)}+\frac{2y}{3(x+y)(1-x)}\geq 2,$$
or equivalently
$$\eta_1(x,y)=3(1-x)(4y+x-1)+(2y)^2-12y(x+y)(1-x)\geq 0.$$
By elementary computations, we can obtain
$$\eta_1(x,y)=4(3x-2)y^2+12(1-x)^2y-3(1-x)^2.$$
Since $0\leq x\leq \frac{1}{3}$, $\eta_1(x,y)$ is a convex quadratic function with respect to $y$. Hence, it suffices to show that
$\eta_1(x,y)\geq 0$, when $x+y=1$ or $y=1$. This is obvious because if $x+y=1$, then we have
$$\eta_1(x,y)=\eta_1(1-y,y)=4-3y^2\geq 1,$$
and if $x=y$, then
$$\eta_1(x,y)=\eta_1(x,x)=(3x-1)^2\geq 0.$$
\item[\textbf{Case 2}] $(x,y)\in B_2$. Note that
$$y^2(1+3x)+(4y+x-1)-4y(x+y)=(1-x)(1-y)(3y-1)\geq 0.$$
This implies that
$$\frac{y(1+3x)}{2(x+y)}+\frac{4y+x-1}{2y(x+y)}\geq 2.$$
\item[\textbf{Case 3}] $(x,y)\in B_3$. We have to show that
$$\frac{2(x-x^2+y-y^2-xy)}{(x+y)(4-4x-4y+3xy)}+\frac{4y+x-1}{2y(x+y)}\geq 2.$$
To do this, we define
\begin{align*}
\eta_3(x,y) &= 4y(x-x^2+y-y^2-xy)+(4y+x-1)(4-4x-4y+3xy)\\
&~~~-4y(x+y)(4-4x-4y+3xy)\\
&=(1-x)((12y^2-15+4)x+4(3y^3-7y^2+5y-1))
\end{align*}
Let
$$\eta_3^*(x,y)=(12y^2-15+4)x+4(3y^3-7y^2+5y-1).$$
It is clear that $\eta_3^*(x,y)$ is a linear function with respect to $x$. Hence, it suffices to show that
$\eta_3^*(x,y)\geq 0$, when $x+y=1$ or $x=y$. If $x+y=1$, then
$$\eta_3^*(x,y)=\eta_3^*(1-y,y)=y(12y+13)(1-y)\geq 0.$$
If $x=y$, then
$$\eta_3^*(x,y)=\eta_3^*(y,y)=12y^3-43y^2+36y-4.$$
By determining the second derivative, one can see that $\kappa(y)=12y^3-43y^2+36y-4$ is a convex function on $[\frac{1}{2},\frac{2}{3}]$. Obviously,
$$\kappa\left(\frac{1}{2}\right)=\frac{19}{4}$$
and
$$\kappa\left(\frac{2}{3}\right)=\frac{40}{9}.$$
Therefore $\eta_3^*(y,y)=\kappa(y)\geq \frac{40}{9}>0$.
\end{enumerate}
In conclusion, we have
\begin{equation}\label{harmonic inequality}
\frac{1}{\delta_{L}(x,y)}+\frac{1}{\vartheta_{L}(x,y)}\geq 2,
\end{equation}
for all $(x,y)\in D$.

This completes the proof of the second part of Corollary \ref{main cor}. Furthermore, by elementary computations, one can show that when $\frac{2}{3}\leq u\leq 1\leq v\leq \frac{3}{2}$ and $\frac{1}{u}+\frac{1}{v}\geq 2$, we have
$$v\leq 1+\frac{5}{4}\sqrt{1-u}.$$
It is known \cite{fary} that
$$\frac{2}{3}\leq \delta_{L}(K)\leq 1\leq \vartheta_{L}(K)\leq \frac{3}{2},$$
for every convex disk $K$. Hence
$$\frac{2}{3}\leq \delta_{L}(x,y)\leq 1\leq \vartheta_{L}(x,y)\leq \frac{3}{2},$$
for all $(x,y)\in D$.
Therefore, by (\ref{harmonic inequality}), it immediately follows that
$$\vartheta_{L}(x,y)\leq 1+\frac{5}{4}\sqrt{1-\delta_{L}(x,y)},$$
for all $(x,y)\in D$.
\begin{figure}[!ht]
  \centering
    \includegraphics[scale=0.5]{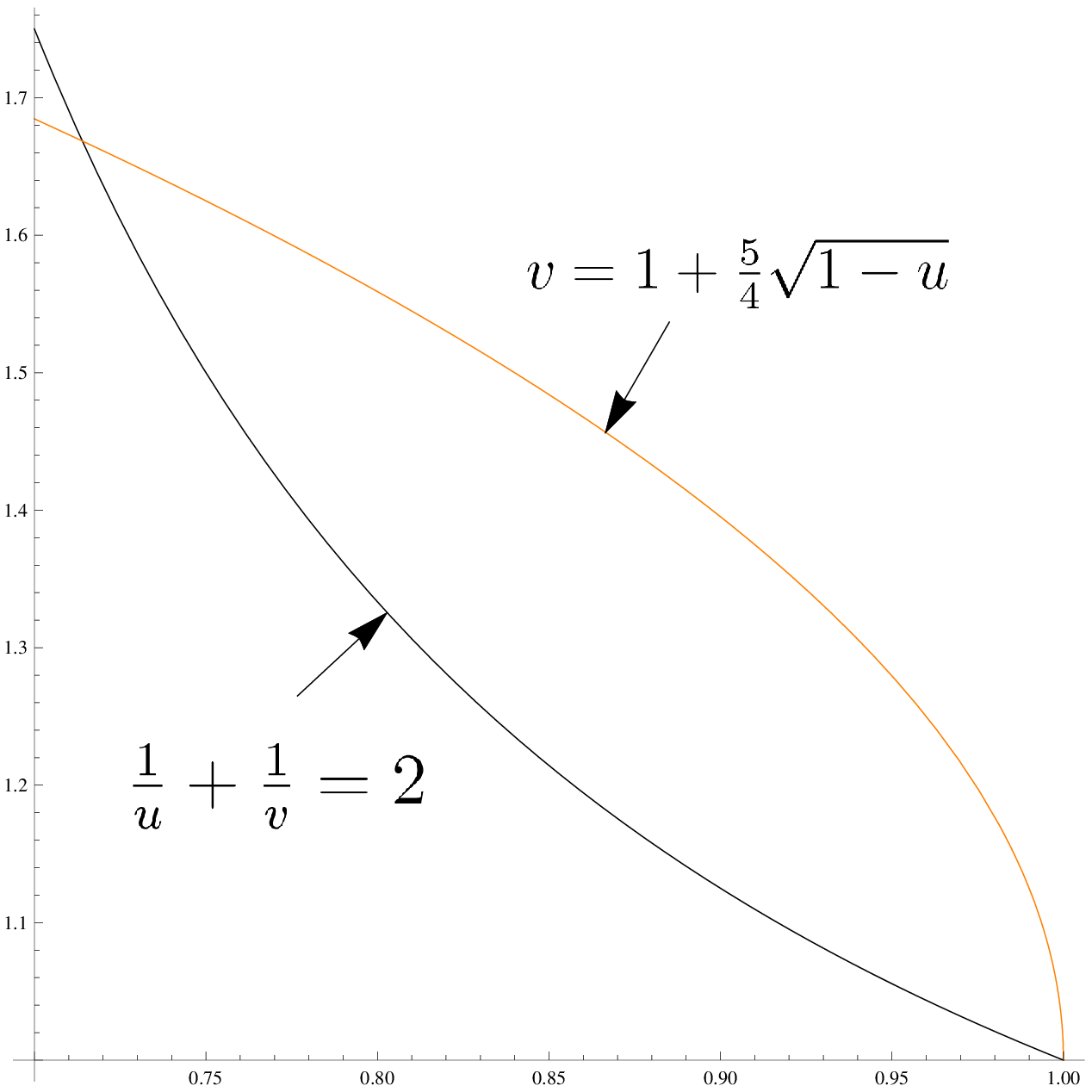}
    \caption{the curves $\frac{1}{u}+\frac{1}{v}=2$ and $v=1+\frac{5}{4}\sqrt{1-u}$}
\end{figure}

\begin{rem}
Let $\mathcal{Q}$ denote the collection of all convex quadrilaterals. We recall that $\omega_{L}$ is defined by $\omega_{L}(K)=(\delta_{L}(K),\vartheta_{L}(K))$ for every $K\in\mathcal{K}^2$. We can determine $\omega_{L}(\mathcal{Q})$. To do this, we divide $\mathcal{Q}$ into three sets
$$\mathcal{Q}_i=\{Q : Q \text{~is affinely equivalent to~}K_{x,y} \text{~for some~}(x,y)\in B_i \},$$
and let $\Omega_i=\omega_{L}(\mathcal{Q}_i)$, where $i=1,2,3$. By using computer calculations, we can obtain $\Omega_1$, $\Omega_2$ and $\Omega_3$ , as shown in Figure \ref{omega1}, Figure \ref{omega2} and Figure \ref{omega3}.
\begin{figure}[!ht]
  \centering
    \includegraphics[scale=0.65]{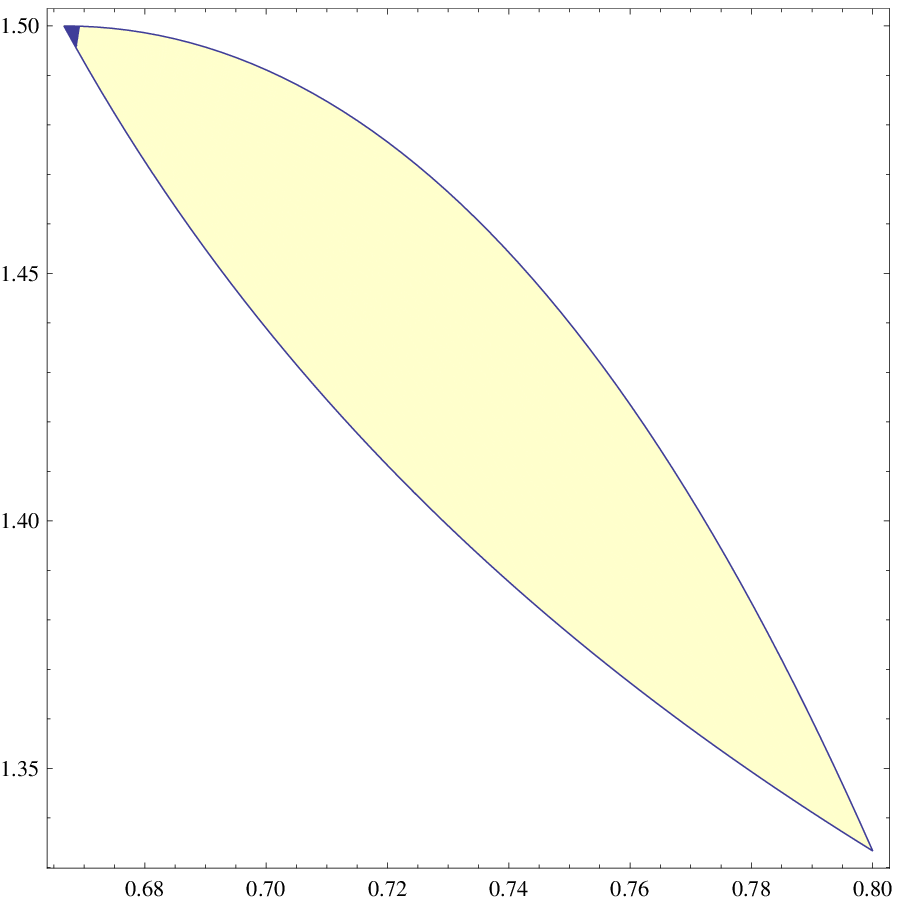}
    \caption{$\Omega_1$}
    \label{omega1}
\end{figure}

\begin{figure}[!ht]
  \centering
    \includegraphics[scale=0.65]{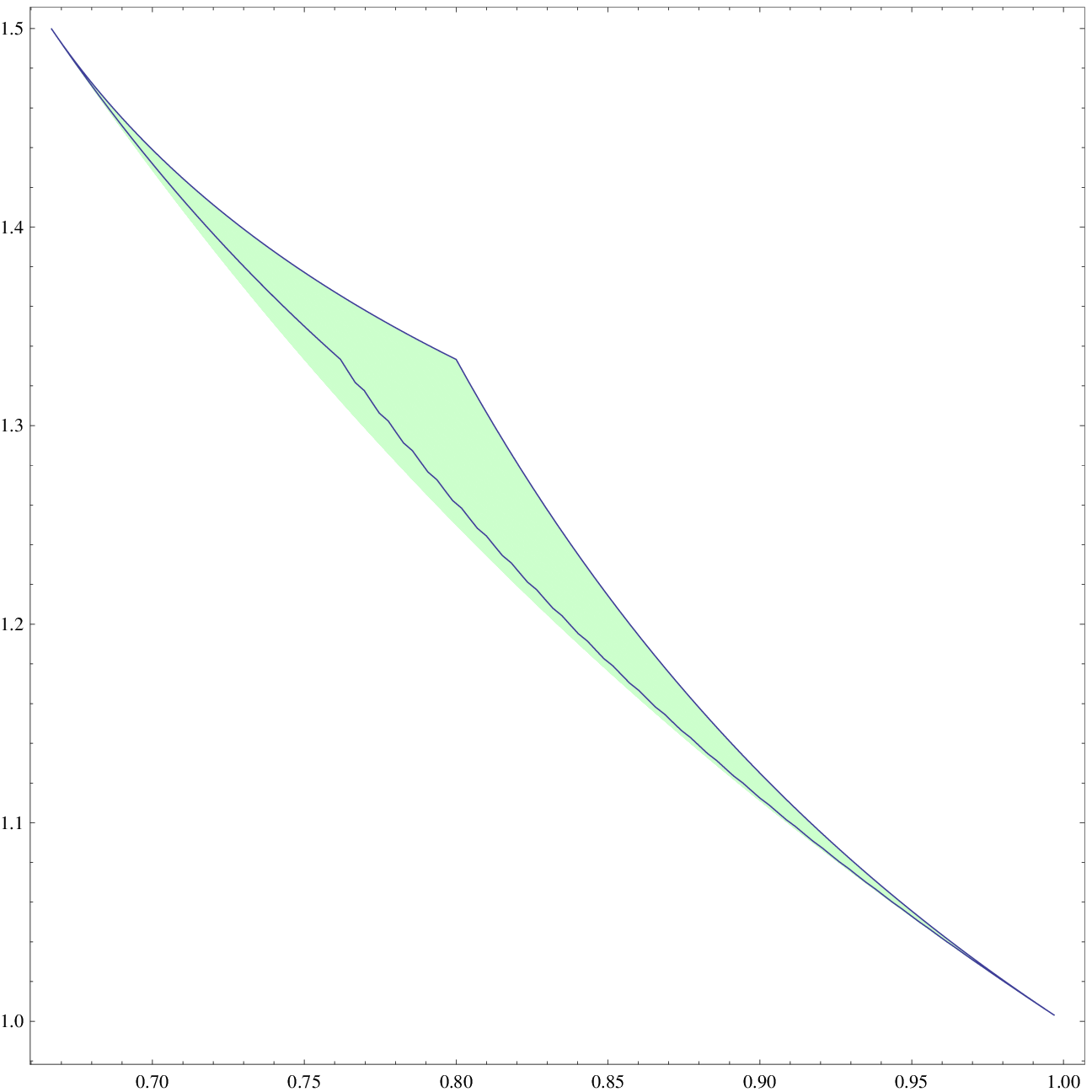}
    \caption{$\Omega_2$}
    \label{omega2}
\end{figure}

\begin{figure}[!ht]
  \centering
    \includegraphics[scale=0.65]{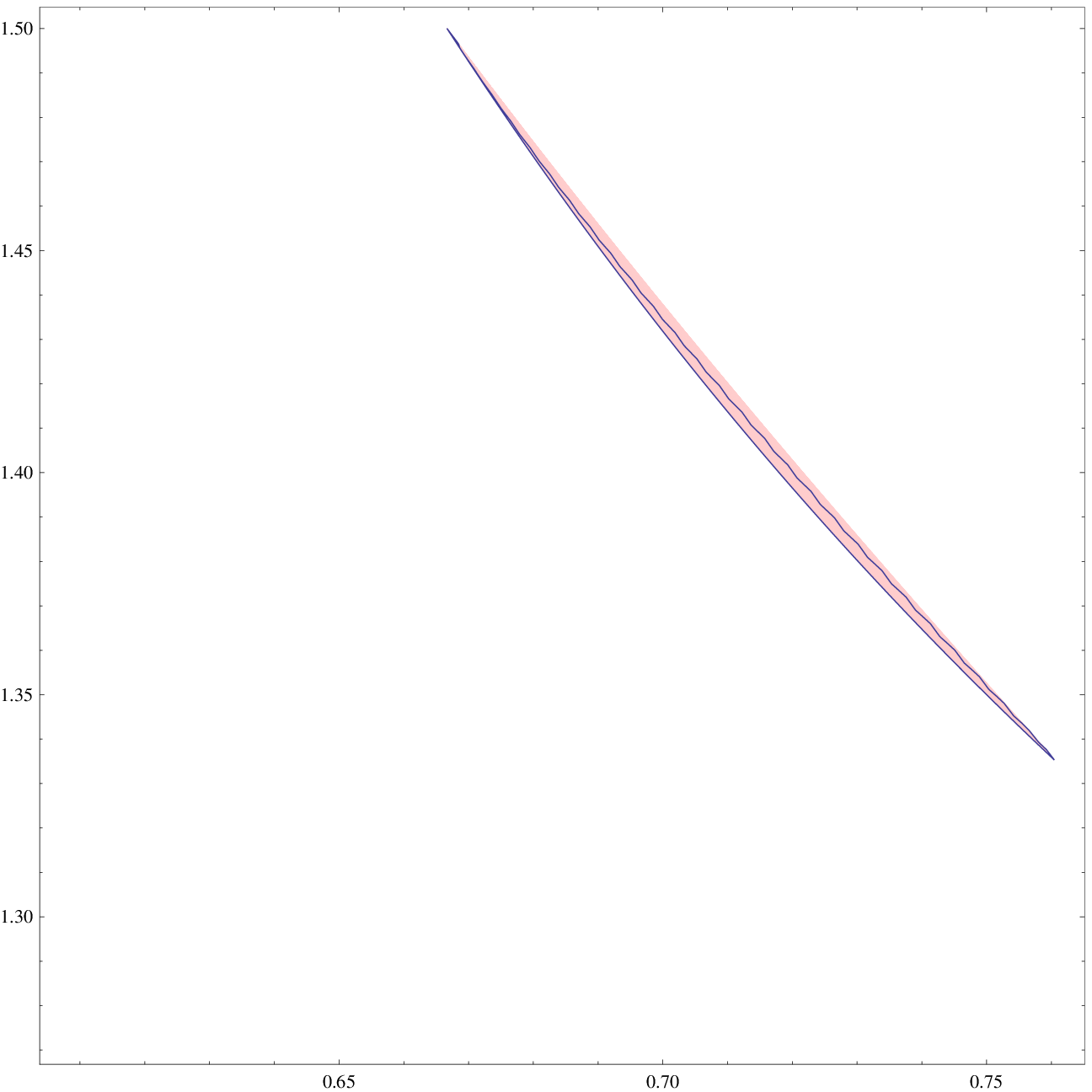}
    \caption{$\Omega_3$}
    \label{omega3}
\end{figure}
Corollary \ref{main cor} can be expressed as : $\omega_{L}(\mathcal{Q})=\Omega_1\cup\Omega_2\cup\Omega_3$ lies between the curves $uv=1$ and $\frac{1}{u}+\frac{1}{v}=2$, as shown in Figure \ref{omega123}.

\begin{figure}[!ht]
  \centering
    \includegraphics[scale=0.5]{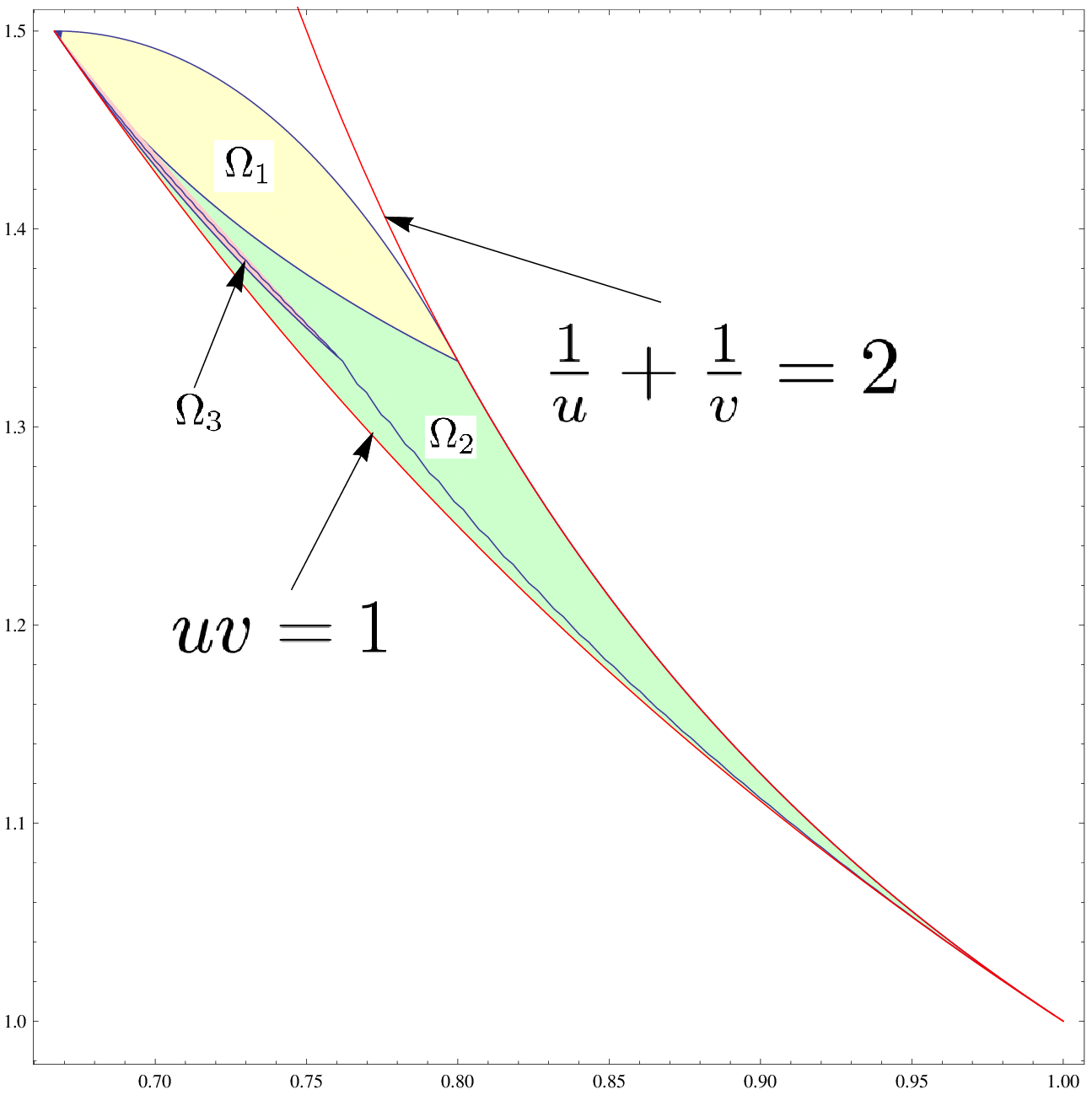}
    \caption{$\omega_{L}(\mathcal{Q})$}
    \label{omega123}
\end{figure}

\end{rem}
\begin{rem}
Let $K_f$ be the convex disk defined in Section \ref{intro}. We note that there exist convex disks $K_f$ such that the inequality $\frac{1}{\delta_{L}(K_f)}+\frac{1}{\vartheta_{L}(K_f)}\geq 2$ dose not hold. For example, we can let $f(x)=1-x^3$. By Theorem \ref{deltawitharea} and Theorem \ref{transequallattice_witharea}, one can get
$$\delta_{L}(K_f)=0.8384279476\ldots$$
and
$$\vartheta_{L}(K_f)=1.282632608\ldots$$
and hence
$$\frac{1}{\delta_{L}(K_f)}+\frac{1}{\vartheta_{L}(K_f)}=1.972354815\ldots < 2$$
\end{rem}

\section{Problems}
Let $K_f$ be the convex disk defined in Section \ref{intro}. We would like to conclude with two questions

\medskip
\noindent
\textbf{Problem 1} Dose the inequality
$$\delta_{L}(K_f)\vartheta_{L}(K_f)\geq 1$$
hold for every convex disk $K_f$?

\medskip
\noindent
\textbf{Problem 2} Is it true that
$$\delta_{L}(K_f)+\vartheta_{L}(K_f)\geq 2$$
and
$$\vartheta_{L}(K_f)\leq 1+\frac{5}{4}\sqrt{1-\delta_{L}(K_f)}$$
for every convex disk $K_f$?

\section*{Acknowledgement}
I am very much indebted to Professor ChuanMing Zong for his valuable suggestions and many interesting discussions on the topic. This work is supported by 973 programs 2013CB834201 and 2011CB302401.

%%%%%%%%%%%%%%%%%%%%%%%%%%%

\end{document}